\documentclass[a4paper]{amsart}

\usepackage[utf8]{inputenc}			
\usepackage[T1]{fontenc}			
\usepackage[british]{babel} 		

\usepackage{
amsmath,			
amsthm,				
amssymb,			
latexsym,			
xfrac,				
faktor,				
mathtools,			
bbm,				
mathrsfs,			
enumitem,			
microtype,			
hyphenat,			
hyperref,			
cite,				
cleveref			
}

\newtheorem{theo}{Theorem}
\newtheorem{prop}[theo]{Proposition}
\newtheorem{lem}[theo]{Lemma}
\newtheorem{coro}[theo]{Corollary}
\newtheorem*{fact}{Fact}
\newtheorem*{claim}{Claim}

\theoremstyle{definition}

\newtheorem{rmk}[theo]{Remark}

\crefname{theo}{Theorem}{Theorems}
\crefname{lem}{Lemma}{Lemmas}
\crefname{coro}{Corollary}{Corollaries}
\crefname{defi}{Definition}{Definitions}
\crefname{ex}{Example}{Examples}
\crefname{rmk}{Remark}{Remarks}
\crefname{enumi}{}{}

\renewcommand\qedsymbol{\setbox0=\hbox{\ \ \ \footnotesize{\normalfont Q.E.D.}}\kern\wd0 \strut \hfill \kern-\wd0 \box0}
\newcommand*\claimqed{\renewcommand\qedsymbol{\setbox0=\hbox{\ \ \ {\normalfont \ensuremath{\square}}}\kern\wd0 \strut \hfill \kern-\wd0 \box0}\qedhere\renewcommand\qedsymbol{\setbox0=\hbox{\ \ \ \footnotesize{\normalfont Q.E.D.}}\kern\wd0 \strut \hfill \kern-\wd0 \box0}}

\newcommand*\N{\mathbb{N}}
\newcommand*\R{\mathbb{R}}
\newcommand*\D{\mathbb{D}}

\newcommand*\tbullet{\text{\raisebox{0.25pt}{\scalebox{0.6}{$\bullet$}}}}
\newcommand*\sth{\mathrel{:}}
\newcommand*\map{\colon}

\newcommand*\tp{\mathrm{tp}}
\newcommand*\inc{\iota}

\newcommand*\quot{\mathrm{quot}}
\newcommand*\Mm{\mathrm{M}}
\newcommand*\Nn{\mathrm{N}}
\newcommand*\dd{\mathrm{d}}
\newcommand*\diam{\mathrm{diam}}
\newcommand*\cov{\mathrm{cov}}
\newcommand*\Vol{\mathrm{Vol}}
\newcommand*\st{\mathrm{st}}
\newcommand*\medium{\mathfrak{m}}
\newcommand*\uu{\mathfrak{u}}
\newcommand*\lang{\mathtt{L}}

\title{On metric approximate subgroups}

\author{Ehud Hrushovski}
\thanks{The first author declares that part of the work was carried out in the Hebrew University of Jerusalem with support by the European Research Council under the European Union Seventh Framework Program (FP7/2007-2013): ERC Grant Agreement No. 291111.} 
\address[E.~Hrushovski]{University of Oxford; Mathematical Institute; Andrew Wiles Building, Woodstock Road, Oxford, OX2 6GG, England}
\email[E.~Hrushovski]{Ehud.Hrushovski@maths.ox.ac.uk}

\author{Arturo Rodr\'{\i}guez Fanlo}
\thanks{The second author declares that his work was carried out in the University of Oxford supported by a PhD studentship from the Mathematical Institute funded by the Engineering and Physical Sciences Research Council: EPSRC Dept. Award 1.22.}
\address[A.~Rodr\'{\i}guez Fanlo]{Hebrew University of Jerusalem; Einstein Institute of Mathematics; Jerusalem, 9190401, Israel}
\email[A.~Rodr\'{\i}guez Fanlo]{arturo.rodriguez@mail.huji.ac.il}


\date{\today}

\begin{document}

\begin{abstract}
Let $G$ be a group with a metric invariant under left and right translations, and let $\D_r$ be the ball of radius $r$ around the identity. A $(k,r)${\hyp}metric approximate subgroup is a symmetric subset $X$ of $G$ such that the pairwise product set $XX$ is covered by at most $k$ translates of $X\D_r$. This notion was introduced in \cite{tao2008product} and \cite{tao2014metric} along with the version for discrete groups (approximate subgroups). In \cite{hrushovski2012stable}, it was shown for the discrete case that, at the asymptotic limit of $X$ finite but large, the ``approximateness'' (or need for more than one translate) can be attributed to a canonically associated Lie group. Here we prove an analogous result in the metric setting, under a certain finite covering assumption on $X$ replacing finiteness. In particular, if $G$ has bounded exponent, we show that any $(k,r)${\hyp}metric approximate subgroup is close to a $(1,r')${\hyp}metric approximate subgroup for an appropriate $r'$.
\end{abstract}

\maketitle

\section*{Introduction}
Approximate subgroups were defined in \cite{tao2008product} along with a metric version, \emph{rough approximate subgroups}, allowing also a thickening of the set (see also \cite{tao2014metric}). Two subsets of a group are \emph{$(k,T)${\hyp}roughly commensurable} if $k$ translates of one of them cover the other up to thickening by $T$ (i.e. multiplying by $T$ on the right), where $T$ is a subset containing the identity. A \emph{$(k,T)${\hyp}rough approximate subgroup} is a symmetric subset containing the identity which is $(k,T)${\hyp}roughly commensurable to its pairwise product set. We simply say that two set are $k${\hyp}commensurable if they are $(k,\boldsymbol{1})${\hyp}roughly commensurable and, similarly, we say that a set is a $k${\hyp}approximate subgroup if it is a $(k,\boldsymbol{1})${\hyp}rough approximate subgroup, where $\boldsymbol{1}\coloneqq \{1\}$. As long as there is no confusion, to simplify the notation, we may omit the parameters $k$ or $T$.

Finite approximate subgroups arose in analysis (see \cite{freiman1973foundations}), and have had deep significance in group theory and additive combinatorics. See especially the paper \cite{breuillard2012structure} that classifies finite approximate subgroups, and the many papers that cite it. There are also connections to basic model theoretic concepts such as connected components and strong types, as discussed, for example, in \cite{hrushovski2022amenability}. The basic fact is that, in the presence of an invariant measure, approximate subgroups are governed by canonically associated Lie groups \cite{hrushovski2012stable}; it is this that we aim to generalise to the metric case.

By a metric group we mean a group together with a metric invariant under left translations. We will write $\D_r(1)$ for the closed ball of radius $r$ at the identity and, similarly, $\D_r(X)\coloneqq X\D_r(1)$ for the (closed) $r${\hyp}thickening of $X$. In this case, we say that two sets are \emph{$(k,r)${\hyp}commensurable} if they are $(k,\D_r(1))${\hyp}roughly commensurable and that a set is a \emph{$(k,r)${\hyp}metric approximate subgroup} if it is a $(k,\D_r(1))${\hyp}rough approximate subgroup.

Our main theorem states, roughly speaking, that an infinitesimal thickening of any ultraproduct $X=\faktor{\prod X_m}{\uu}$ of symmetric subsets of metric groups having \textsl{regular enough discretisations} is \textsl{nicely associated} to a compact neighbourhood of the identity of a Lie group. These two notions will be made precise in the second and third sections, respectively (for the latter, see also the definition of Lie models in \cite{rodriguez2020piecewise}). 

The proof of \cref{t:main theorem} will require left invariance along with a local Lipschitz regularity for right translations. Here, to simplify the exposition, we consider first \emph{bi{\hyp}invariant metric groups}, i.e. metric groups where also right translations are isometries. We can then give the following statement:

\begin{theo}[Metric Lie Model] \label{t:main theorem easy} Let $(G_m,X_m,r_{i,m})_{i\leq m\in\N}$ be a sequence such that, for some fixed $k\in\N$,
\begin{enumerate}[label={\emph{(\alph*)}}, wide]
\item $G_m$ is a bi{\hyp}invariant metric group,
\item $X_m$ is a symmetric subset, 
\item $(r_{i,m})_{i\leq m}$ is a sequence of positive reals with $2r_{i,m}\leq r_{i-1,m}$ and
\[\Nn_{r_{i,m}}(X^9_m)\leq k\cdot \Nn_{9r_{i,m}}(X_m)<\infty .\]
\end{enumerate}
Let $(G^*,X,\ldots)$ be a non{\hyp}principal ultraproduct with enough structure, $G$ the subgroup generated by $X$ and $o_r(1)=\bigcap_i \D_{r_i}(1)$, where $\D_{r_i}(1)$ is the ultraproduct of the balls of radius $r_{i,m}$. Then, $G\cdot o_r(1)\leq G^*$ has a connected Lie model $\pi\map H\to L=\faktor{H}{K}$ with $o_r(1)\trianglelefteq K\subseteq X^8\cdot o_r(1)$ such that
\begin{enumerate}[label={\emph{(\arabic*)}}, wide]
\item $H\cap X^4\cdot o_r(1)$ and $X^2\cdot o_r(1)$ are commensurable,
\item $H\cap X^8\cdot o_r(1)$ generates $H$, and
\item $\pi[H\cap X^4]$ is a compact neighbourhood of the identity of $L$.
\end{enumerate}
\end{theo}

As a consequence, we get the following corollaries for metric approximate subgroups. 

\begin{coro} \label{c:corollary 1} Fix $k,C,n,s\in\N$. There is $m=m(k^8C,n,s)\in\N$ such that the following holds:\smallskip

Let $G$ be a bi{\hyp}invariant metric group and $X$ a $(k,2^{-m})${\hyp}metric approximate subgroup. Assume \[\Nn_{2^{-m}}(X)\leq C^m\cdot \Nn_1(X)<\infty.\] Then, there are $I\subseteq \{1,\ldots,m\}$ with $|I|=n$ and a symmetric subset $Y$ with $x^{-1}Y^nx\subseteq \D_{2^{-s}}(X^4)$ for all $x\in X^n$ such that $Y^n\subseteq \D_{2^{-s}}(X^4)$ and 
\[\Nn_{2^{-i}}(Y)\geq \frac{1}{m} \Nn_{2^{-i}}(X) \mathrm{\ for\ all\ }i\in I.\]
\end{coro}

\begin{coro} \label{c:corollary 2} Fix $k,C,N,n,s\in\N$. There are $c\coloneqq c(k^8C,N)\in\N$ and $m\coloneqq m(k^8C,N,n,s)\in\N$ such that the following holds:\smallskip

Let $G$ be a bi{\hyp}invariant metric group and $X$ a $(k,2^{-m})${\hyp}metric approximate subgroup. Assume that $\dd(g^N,1)\leq 2^{-m}$ for every $g\in X^8$ and \[\Nn_{2^{-m}}(X)\leq C^m\cdot\Nn_1(X)<\infty.\] Then, there is a symmetric subset $Y\subseteq X^{16}$ such that $Y$ and $X^2$ are $(c,2^{-s})${\hyp}commensurable and $Y^n\subseteq \D_{2^{-s}}(Y)$.
\end{coro}

\begin{coro} \label{c:corollary 3} Fix $k,C,N\in\N$ and any function $s\map \N\to \N$. There is $m\coloneqq m(k^8C,N,s)\in\N$ such that the following holds:\smallskip

Let $G$ be a bi{\hyp}invariant metric group and $X$ a $(k,2^{-m})${\hyp}metric approximate subgroup. Assume \[\Nn_{2^{-m}}(X)\leq C^m\cdot \Nn_1(X)<\infty.\] Then, for some $c\leq m$, there is a sequence $X_N\subseteq \cdots\subseteq X_0\subseteq X^8$ satisfying the following properties for each $n\leq N$ with $s=s(c)$: 
\[\begin{array}{ll}
\mathrm{(1)}& X^2\mathrm{\ and\ }X_1\mathrm{\ are\ }(c,2^{-s})\mbox{\hyp}\mathrm{commensurable.}\\ 
\mathrm{(2)} & X_{n+1}X_{n+1}\subseteq \D_{2^{-s}}(X_n).\\
\mathrm{(3)} & X_n\mathrm{\ is\ covered\ by\ }c\mathrm{\ translates\ of\ }\D_{2^{-s}}(X_{n+1}).\\
\mathrm{(4)} & x^{-1}X_{n+1}x\subseteq \D_{2^{-s}}(X_n)\mathrm{\ for\ every\ }x\in X_1.\\
\mathrm{(5)} & [X_{n_1},X_{n_2}]\subseteq \D_{2^{-s}}(X_n)\mathrm{\ whenever\ }n< n_1+n_2.\\
\mathrm{(6)} & \{x\in X_0\sth x^{17}\in X_n\}\subseteq X_{n+1}.\\
\mathrm{(7)} &\mathrm{If\ }x,y\in X_0\mathrm{\ with\ }x^2=y^2,\mathrm{\ then\ }y^{-1}x\in \D_{2^{-s}}(X_N).\end{array}\]
\end{coro}

The notion of ultraproduct appropriate to metric groups is similar to their use in continuous logic: from an ordinary ultraproduct of metric groups, we retain only the elements at bounded distance from the identity, and then quotient out by the infinitesimals. This involves what is known as a \emph{piecewise hyperdefinable set} in model theory. A systematic theory of piecewise hyperdefinable sets was developed in \cite{rodriguez2020piecewise}, following an ad hoc treatment in an early version of this note. Here, we briefly recall in the third section the main results from \cite{rodriguez2020piecewise} that we will need.

With this in place, we can form the quotient by the infinitesimals $o_r(1)$ and check that $\faktor{X}{o_r(1)}$ satisfies the conditions needed to apply \cref{t:rough lie model}. In particular, we need to show that $X^2$ is a $o_r(1)${\hyp}rough approximate subgroup and that $\faktor{X}{o_r(1)}$ is a near{\hyp}subgroup, i.e. it has an ideal of subsets satisfying various model{\hyp}theoretic properties. 

The Lie model results of \cite{hrushovski2012stable} were obtained using the measure{\hyp}zero ideal of an appropriate measure, as well as other ideals with similar properties. We will make use of this additional flexibility here. By contrast Massicot and Wagner in \cite{massicot2015approximate}, following Sanders as presented in \cite{breuillard2012structure}, make use of the actual measure. While we also start with a measure, we construct an ideal from it in a different way, and then work with the ideal; this will be explained in the third section.

By comparison to the vast field of work on approximate subgroups, there has been little work on the metric case that we are aware of; but two very striking results have appeared which we now describe:

In the abstract setting, there is an extremely interesting foundational work by Gowers and Long \cite{gowers2020partial} regarding approximate associativity in partial Latin squares, that has metric subgroups in its conclusion. It would be very interesting to try to connect a result like ours with their work, to see under what conditions \cite{gowers2020partial} can be improved to give a Lie group rather than an abstract metric group.

There has also been a vast field of work on finite approximate subgroups of linear groups, due to Erd\"{o}s{\hyp}Szemer\'{e}di, Bourgain{\hyp}Katz{\hyp}Tao, Helfgott, Pyber{\hyp}Szab\'{o}, Breuillard{\hyp}Green{\hyp}Tao, and others. In this line, a strong metric version was proved by de Saxc\'{e} \cite{desaxce2014product}. We further discuss the relation of our work with it at the end of this paper.

\

The first section reviews basic properties of discretisation in a metric setting. The second section recalls the model theoretic concepts needed for this paper from \cite{rodriguez2020piecewise}. The ideal of $\bigwedge${\hyp}definable sets that we will use is constructed in Section 3. The fourth section is the proof of the main theorem. We conclude proving the indicated corollaries. 

\

\subsection*{Notations}

\begin{enumerate}[label={\raisebox{0.4ex}{\rule{0.4em}{0.4em}}}, itemsep=2pt, wide]
\item Let $R\subseteq X\times Y$ be a set{\hyp}theoretic binary relation. For $V\subseteq X$, we write $R(V)\coloneqq\{y\in Y\sth \exists x\in V\ (x,y)\in R\}$. We also write $R^{-1}\coloneqq\{(y,x)\sth (x,y)\in R\}$, so $R^{-1}(y)\coloneqq \{x\in X\sth (x,y)\in R\}$ for $y\in Y$ and $R^{-1}(W)\coloneqq \{x\in X\sth \exists y\in W\ (x,y)\in R\}$ for $W\subseteq Y$. This notation is also used for partial functions, which are always identified with their graphs.
\item The cardinality of a set $X$ is denoted by $|X|$. The set of natural numbers with $0$ is denoted by $\N$. We also write $\omega\coloneqq \N\coloneqq \aleph_0$. 
\item Here, definable means definable with parameters. For a set of parameters $A$, we write $\bigwedge_A${\hyp}definable for type{\hyp}definable over $A$. Similarly, $\bigvee_A${\hyp}definable means cotype{\hyp}definable over $A$. We also use cardinals and cardinal inequalities. In that case, the subscript should be read as an anonymous set of parameters whose size satisfies the indicated condition. For example, $\bigwedge_{\omega}${\hyp}definable means type{\hyp}definable over a countable set of parameters. 
\item We use product notation for groups. Also, unless otherwise stated, we consider the group acting on itself on the left. In particular, by a translate we mean a left translate. For subsets $X$ and $Y$ of a group, we write $XY$ for the set of pairwise products. We abbreviate $X^n\coloneqq XX^{n-1}$ and $X^{-n}\coloneqq (X^{-1})^n$ for $n\in\N$. A subset $X$ of a group is called \emph{symmetric} if $1\in X=X^{-1}$. For subsets $X$ and $Y$ of a group, we write $Y^X\coloneqq \{y^x\sth x\in X,\ y\in Y\}$ where $y^x\coloneqq x^{-1}yx$. We say that $X$ \emph{normalises} $Y$ if $Y^X\subseteq Y$. For subsets $X$ and $Y$ of a group, we write $[X,Y]\coloneqq \{[x,y]\sth x\in X,\ y\in Y\}$ where $[x,y]\coloneqq x^{-1}y^{-1}xy$.
\end{enumerate}

\section{Discretisations}
Through this section we work on a metric space with metric $\dd$. For $r\in \R_{\geq 0}$, we define the \emph{(closed) $r${\hyp}distance relation} as $\D_r\coloneqq \{(x,y)\sth \dd(x,y)\leq r\}$. For a subset $X$, the \emph{(closed) $r${\hyp}thickening of $X$} is the set $\D_r(X)=\{y\sth \dd(x,y)\leq r\ \mathrm{for\ some\ }x\in X\}$. In particular, for a point $x$, $\D_r(x)$ is the closed ball of radius $r$ at $x$. 

Recall that a set $Z$ is \emph{$r${\hyp}separated} if $\dd(z,z')>r$ for every $z,z'\in Z$ with $z\neq z'$. An \emph{$r${\hyp}discretisation} of $X$ is an $r${\hyp}separated finite subset $Z\subseteq X$ of maximal size. The \emph{$r${\hyp}discretisation number of $X$} is the size of the $r${\hyp}discretisations of $X$, i.e. the maximum size $\Nn_r(X)$ of a finite $r${\hyp}separated subset of $X$. We write $\Nn_r(X)=\infty$ when there are no $r${\hyp}discretisations of $X$. 

A \emph{covering $r${\hyp}discretisation of $X$ relative to $Y$} is a finite subset $Z\subseteq Y$ of minimal size such that $X\subseteq \D_r(Z)$. The \emph{covering $r${\hyp}discretisation number of $X$ relative to $Y$} is the size of the covering $r${\hyp}discretisations of $X$ relative to $Y$, i.e. the minimum size $\Nn^\cov_r(X/Y)$ of a finite subset $Z$ of $Y$ such that $X\subseteq \D_r(Z)$. We write $\Nn^\cov_r(X/Y)=\infty$ if $X$ has no covering $r${\hyp}discretisations relative to $Y$.

We recall now some basic properties of discretisation numbers and covering discretisation numbers in \cref{l:subadditivity and monotonicity of discretisation numbers,l:continuity of discretisation numbers,l:equivalence discretisation numbers}. We omit the easy proofs, leaving their elaboration to the readers.

Discretisations and covering discretisations are not unique but their numbers are indeed unique. Thus, a natural way of measuring the discreteness of a set consists on studying its (covering) discretisation numbers. Our first remark is that, although the functions $\Nn_r(X)$ and $\Nn^\cov_r(X/Y)$ are not measures, they still enjoy one of the main properties of measures. Recall that a function $\nu$ on sets is \emph{subadditive} if $\nu(A\cup B)\leq \nu(A)+\nu(B)$ for any $A$ and $B$. 

\begin{lem} \label{l:subadditivity and monotonicity of discretisation numbers} The functions $\Nn_r(X)$ and $\Nn^\cov_r(X/Y)$ are subadditive on $X$. Also, $\Nn_r(X)$ is increasing on $X$ and decreasing on $r$, and $\Nn^\cov_r(X/Y)$ is increasing on $X$ and decreasing on $r$ and $Y$. Obviously, $\Nn_r(X)=0$ if and only if $\Nn^\cov_r(X/Y)=0$, if and only if $X=\emptyset$.
\end{lem}
Furthermore, they are continuous according to their monotony:

\begin{lem} \label{l:continuity of discretisation numbers} For any $X,X_1,X_2,\ldots$ and any $r\in\R_{\geq 0}$, 
\[\Nn_r(X)=\sup_{\varepsilon>0}\Nn_{r+\varepsilon}(X)\mathrm{\ and\ }\Nn_r({\textstyle\bigcup^{\infty} X_n})=\sup_{m\in\N}\Nn_r({\textstyle\bigcup^{m}X_n}).\]

For any $X,X_1,X_2,\ldots$, any $r\in\R_{\geq 0}$ and any $Y$ compact in the metric topology, 
\[\Nn^\cov_r(X/Y)=\sup_{\varepsilon>0}\Nn^\cov_{r+\varepsilon}(X/Y)\mathrm{\ and\ }\Nn^\cov_r({\textstyle\bigcup^{\infty} X_n}/Y)=\sup_{m\in\N}\Nn^\cov_r({\textstyle\bigcup^{m} X_n}/Y).\]
\end{lem}
\begin{rmk} \cref{l:subadditivity and monotonicity of discretisation numbers,l:continuity of discretisation numbers} together imply in particular that $\Nn_r(\tbullet)$ and $\Nn^\cov_r(\tbullet/Y)$ are outer{\hyp}measures for any $r\in\R_{\geq 0}$ and $Y$ compact.
\end{rmk} 
Covering discretisations are extrinsic, i.e. they depend on the ambient space $Y$. Therefore, it is natural to work preferably with discretisations than with covering discretisations. Despite this consideration, the choice is still arbitrary. Indeed, the following lemma shows that they are in fact essentially equivalent up to a constant factor of $2$ in the scale.
\begin{lem} \label{l:equivalence discretisation numbers} For any $X\subseteq Y$ and $r\in\R_{>0}$,
$\Nn_{2r}(X)\leq \Nn^\cov_r(X/Y)\leq \Nn_r(X).$
\end{lem}


A \emph{(left) metric group} is a group together with a metric invariant under left translations. Note that, in this case, $\D_r(X)=X\D_r(1)$. We prove now some basic lemmas that we will need for our main \cref{t:main theorem}.
\begin{lem} \label{l:lemma 1 application to metric groups} Let $G$ be a metric group, $X\subseteq G$, $m\in\N_{\geq 2}$ and $r,k\in\R_{\geq 0}$. Assume that \[\Nn_r(XX^{-1}X)\leq k\cdot \Nn_{(2m+1) r}(X)<\infty.\] Then:
\begin{enumerate}[label={\textnormal{(\arabic*)}}, wide]
\item \label{itm:lemma 1 application to metric groups:1} For every $b\in X$, $\Nn_{r}(X\cap \D_{mr}(b))\leq k$.
\item \label{itm:lemma 1 application to metric groups:2} For every $Y\subseteq X$, $\Nn^\cov_{r}(Y/X)\leq k\cdot \Nn^\cov_{{mr}}(Y/X)$.
\end{enumerate}
\end{lem}
\begin{proof} 
\begin{enumerate}[label={(\arabic*)}, wide]
\item[\hspace{-1.4em}\setcounter{enumi}{1}\theenumi] Let $Z$ be a $(2m+1)r${\hyp}discretisation of $X$. For each $z\in Z$, let $W_z$ be an $r${\hyp}discretisation of $\D_{mr}(z)\cap XX^{-1}X$. For different $z,z'\in Z$, we have that
\[(2m+1)r< \dd(z,z')\leq \dd(z,w)+\dd(w,w')+\dd(w',z')\leq 2mr+\dd(w,w'),\]
for any $w\in W_z$ and $w'\in W_{z'}$. Therefore, $\dd(W_z,W_{z'})>r$, so $W_z\cap W_{z'}=\emptyset$. Write $n_z\coloneqq |W_z|$ and $n\coloneqq |Z|=\Nn_{(2m+1)r}(X)$. Then,
\[\sum_{z\in Z} n_z=\left|\bigcup_{z\in Z}W_z\right|\leq N_r(XX^{-1}X).\]
In particular, $\min\, n_z\leq \frac{1}{n}\sum n_z\leq k$. Take $z_0\in Z$ with $n_{z_0}=\min\, n_z$. For each $b\in X$, consider $f_b\map X\to XX^{-1}X$ given by $f_b(x)=z_0b^{-1}x$. Clearly, $f_b$ is an isometry. Note $f_b(\D_{mr}(b)\cap X)\subseteq \D_{mr}(z_0)\cap XX^{-1}X$. Also, if $Y\subseteq X$ is $r${\hyp}separated, then $f_b(Y)$ is $r${\hyp}separated. Thus, $\Nn_{r}(\D_{mr}(b)\cap X)\leq \Nn_{r}(\D_{mr}(z_0)\cap XX^{-1}X)=n_{z_0}\leq k$.
\item By definition, $Y$ can be covered by $\Nn^\cov_{{mr}}(Y/X)$ balls of radius ${mr}$ centred at points of $X$. By \ref*{itm:lemma 1 application to metric groups:1} and \cref{l:equivalence discretisation numbers}, each one of these balls restricted to $Y\subseteq X$ can be covered by $k$ balls of radius $r$ centred at points of $X$. Hence, we conclude that $\Nn^\cov_{r}(Y/X)\leq k\cdot \Nn^\cov_{{mr}}(Y/X)$. \qedhere
\end{enumerate}
\end{proof}

\begin{lem} \label{l:lemma 2 application to metric groups} 
Let $G$ be a metric group, $X\subseteq G$ and $r,k\in \R_{\geq 0}$. Assume that
\[\Nn_{r}(XX^{-1}X)\leq k\cdot \Nn_{9r}(X)<\infty.\] 
Let $Z$ be a $2r${\hyp}discretisation of $X$. Then, for any $Y\subseteq X$, 
\[\Nn^\cov_{2r}(Y/X)\leq \left|Z\cap \D_{2r}(Y)\right|\leq k\cdot \Nn^\cov_{2r}(Y/X).\]
\end{lem}

\begin{proof} We know that $X\subseteq \D_{2r}(Z)$. Thus, $Y\subseteq \D_{2r}(Z\cap \D_{2r}(Y))$ for any $Y\subseteq X$. Therefore, $\Nn_{2r}^\cov(Y/X)\leq |Z\cap \D_{2r}(Y)|$.

On the other hand, $Z$ is $2r${\hyp}separated, so there is no closed ball of radius $r$ in which there are two elements of $Z$. In other words, $|Z_0|=\Nn^\cov_r(Z_0/X)$ for any $Z_0\subseteq Z$. In particular, $|Z\cap\D_{2r}(Y)|=\Nn^\cov_{r}(Z\cap\D_{2r}(Y)/X)$. By \cref{l:lemma 1 application to metric groups}(\ref{itm:lemma 1 application to metric groups:2}) with $m=4$, we have $\Nn^\cov_{r}(Z\cap \D_{2r}(Y)/X)\leq k\cdot \Nn^\cov_{4r}(Z\cap \D_{2r}(Y)/X)$. Finally, $\Nn^\cov_{4r}(Z\cap \D_{2r}(Y)/X)\leq \Nn^\cov_{4r}(\D_{2r}(Y)/X)\leq \Nn^\cov_{2r}(Y/X)$, concluding 
\[\Nn^\cov_{2r}(Y/X)\leq |Z\cap \D_{2r}(Y)|\leq k\cdot \Nn_{2r}^\cov(Y/X).\]
\end{proof} 

Let $G$ be a metric group, $X\subseteq G$ and $l,r\in\R_{\geq 0}$. We say that $X$ is \emph{(right) $(l,r)${\hyp}Lipschitz} if every right translation by an element of $X$ is $l${\hyp}Lipschitz when restricted to $\D_r(1)$. 

\begin{lem} \label{l:subgroup of infinitesimals} Let $G$ be a metric group and $r_0,\ldots,r_m\in\R_{\geq 0}$ such that $2r_{i+1}\leq r_i$. Assume that $X$ is $(l,r_0)${\hyp}Lipschitz. Then, $\D_{r_{i+1}}(1)\D_{r_{i+1}}(1)^{-1}\subseteq \D_{r_i}(1)$ for every $i<m$ and $x^{-1}\D_{r_{i+k}}(1)x\subseteq \D_{r_i}(1)$ for every $x\in X$ and $i\leq m-k$, where $k=\lceil\log_2(l)\rceil$. 
\end{lem}
\begin{proof} Obvious from the definitions.
\end{proof}
In the following lemma, to simplify the notation, we write $l^{[n]}\coloneqq\sum^{n-1}_{i=0} l^i$ for $n\in\N$ and $l\in\R_{\geq 0}$.
\begin{lem}\label{l:lemma rough approximate subgroups} Let $G$ be a metric group and $X$ a $(k,\delta)${\hyp}metric approximate subgroup. Assume that $X$ is $(l,r)${\hyp}Lipschitz and take $m\in\N_{>0}$ with $l^{[m-3]}\delta<r$. Take $\Delta\subseteq G$ such that $X^2\subseteq \Delta\D_\delta(X)$. Then, $X^n\subseteq \Delta^{n-1}\D_{l^{[n-2]}\delta}(X)$ for each $n\leq m$. In particular, $\Nn^\cov_{r+l^{[n-2]}\delta}(X^n/Y)\leq k^{n-1}\cdot \Nn^\cov_r(X/Y)$ for any $r\in\R_{\geq 0}$, $Y\subseteq G$ and $n\leq m$.
\end{lem}
\begin{proof} By induction on $n$.
\end{proof}
\begin{lem} \label{l:metric approximate subgroups and discretisation numbers} Let $G$ be a metric group and $X$ an $(l,r)${\hyp}Lipschitz symmetric subset. Suppose that $\Nn_r(X^5)\leq k\cdot \Nn_r(X)<\infty$ with $k\in\N$. Then, $X^2$ is a $(k,2lr)${\hyp}metric approximate subgroup.
\end{lem}
\begin{proof} Let $\Delta\subseteq X^4$ be a maximal set such that $\{a \cdot \D_r(X)\sth a\in\Delta\}$ is a family of pairwise disjoint sets. If $Z$ is an $r${\hyp}discretisation of $X$, then $\Delta Z\subseteq X^5$ is also $r${\hyp}separated. Therefore, $|\Delta|\cdot |Z|=|\Delta Z|\leq k\cdot |Z|$, concluding $|\Delta|\leq k$. As $\Delta$ is maximal, for any $a\in X^4$ there is $b\in \Delta$ such that $a \D_r(X)\cap b\D_r(X)\neq \emptyset$. Then, $X^4\subseteq \Delta X^2\D_{2lr}(1)$, concluding that $X^2$ is a $(k,2lr)${\hyp}metric approximate subgroup.
\end{proof}
\begin{rmk} In the previous lemma, if we only assume $\Nn_r(X^4)\leq k\cdot \Nn_r(X)<\infty$, then we get that $X^3\subseteq \Delta X^2\D_{2lr}(1)$ with $|\Delta|\leq k$. Thus, if $X$ is $(l',2lr)${\hyp}Lipschitz, we get that $X^4\subseteq \Delta X^3\D_{2ll'r}(1)\subseteq \Delta^2X^2\D_{2l(l'+1)r}(1)$, concluding that $X^2$ is a $(k^2,2l(l'+1)r)${\hyp}metric approximate subgroup.
\end{rmk}
\begin{lem} \label{l:lemma sequence of growth in doubling scale} Let $k,n\in \N_{>0}$ and $l\in\R_{\geq 1}$. Take $m\geq 2n\log_2(18(1+l^{[7]}))$. Let $G$ be a metric group and $X$ an $(l,1)${\hyp}Lipschitz symmetric subset. Suppose that $X$ is a $(k,2^{-m})${\hyp}metric approximate subgroup and 
\[\Nn_{2^{-m}}(X)\leq C\cdot\Nn_1(X)<\infty.\]
Then, there are $r_1,\ldots,r_n\in\R_{>0}$ with $2r_{i+1}\leq r_i\leq 1$ for each $i<n$ such that 
\[\Nn_{r_i}(X^9)\leq k^8\sqrt[n]{C}\cdot \Nn_{9r_i}(X).\]
\end{lem} 
\begin{proof} Write $\alpha\coloneqq 2^{m/2n}$, so $\alpha\geq 18(1+l^{[7]})$. We have
\[\frac{\Nn_{\alpha^{-2n}}(X)}{\Nn_{\alpha^{-2n+1}}(X)}\cdots\frac{\Nn_{\alpha^{-1}}(X)}{\Nn_{1}(X)}\leq C,\]
where each factor is at least $1$. Set $I\coloneqq\{i\in \{1,\ldots, 2n\}\sth \frac{\Nn_{\alpha^{-i}}(X)}{\Nn_{\alpha^{-i+1}}(X)}\leq \sqrt[n]{C}\}$, so 
\[\sqrt[n]{C}^{2n-|I|}\leq \frac{\Nn_{\alpha^{-2n}}(X)}{\Nn_{\alpha^{-2n+1}}(X)}\cdots\frac{\Nn_{\alpha^{-1}}(X)}{\Nn_{1}(X)}\leq C,\]
concluding that $|I|\geq n$. Pick $r'_1,\ldots,r'_n$ decreasing in $\{\alpha^{-i}\sth i\in I\}$. Then, using \cref{l:equivalence discretisation numbers,l:lemma rough approximate subgroups}, we get 
\[\begin{aligned}\Nn_{2\left(l^{[7]}\delta+r'_i\right)}(X^9)\leq&\ \Nn^\cov_{l^{[7]}\delta+r'_i}(X^9/G)\leq k^8\cdot \Nn^\cov_{r'_i}(X/G)\leq\\
\leq&\ k^8\cdot \Nn_{r'_i}(X)\leq k^8\sqrt[n]{C}\cdot \Nn_{\alpha r'_i}(X).\end{aligned}\]
As $\alpha r'_i\geq 18\left(l^{[7]}\delta+r'_i\right)$, taking $r_i=2\left(l^{[7]}\delta+r'_i\right)$, we have that $2r_{i+1}\leq r_i\leq 1$ and
\[\Nn_{r_i}(X^9)\leq k^8\sqrt[n]{C}\cdot \Nn_{9r_i}(X).\] 
\end{proof}

\section{Preliminaries on piecewise hyperdefinable sets}
In this section we define the model theoretic concepts that we need for this paper. From now on, fix a countable first{\hyp}order language $\lang$, an $\aleph_1${\hyp}saturated $\lang${\hyp}structure $M$ and a countable set of parameters $A$.\medskip

An \emph{$A${\hyp}hyperdefinable set} is a quotient $P=\faktor{X}{E}$, where $X$ is a $\bigwedge_A${\hyp}definable (i.e. type{\hyp}definable over $A$) set and $E$ is a $\bigwedge_A${\hyp}definable equivalence relation. A \emph{$\bigwedge_A${\hyp}definable subset} of $P$ is a subset whose preimage by the canonical quotient map of $P$ is $\bigwedge_A${\hyp}definable. Cartesian products of hyperdefinable sets are naturally hyperdefinable, so we say that a binary relation between hyperdefinable sets is \emph{$\bigwedge_A${\hyp}definable} if it is so as subset of the product. We apply this in particular for partial functions. 

A \emph{piecewise $A${\hyp}hyperdefinable set} is a direct limit $P=\underrightarrow{\lim}\, P_i$ of $A${\hyp}hyperdefinable sets with $\bigwedge_A${\hyp}definable $1${\hyp}to{\hyp}$1$ transition maps $\psi_{ji}\map P_i\to P_j$ --- to simplify the situation, from now on we assume that the index set $I$ of the direct limit is countable. The \emph{pieces} of $P$ are the subsets $P_i$. A \emph{piecewise $\bigwedge_A${\hyp}definable subset} $V\subseteq P$ is a subset such that $V\cap P_i$ is $\bigwedge_A${\hyp}definable for each $i$. A \emph{$\bigwedge_A${\hyp}definable} subset is a piecewise $\bigwedge_A${\hyp}definable set fully contained in one piece. The \emph{$A${\hyp}logic topology} of $P$ is the one given by taking as closed subsets the piecewise $\bigwedge_A${\hyp}definable subsets. The \emph{global logic topology} of $P$, if it is well defined, is the largest logic topology. Note that the global logic topology is well defined if and only if $P$ is bounded (i.e. there is a bound on the cardinality of $P$ true for every elementary extension). In that case, the global logic topology is the one given by taking as closed subsets all the piecewise $\bigwedge${\hyp}definable subsets.

Again, cartesian products of piecewise hyperdefinable sets are naturally piecewise hyperdefinable. A binary relation $R$ between hyperdefinable sets $P=\underrightarrow{\lim}\, P_i$ and $P=\underrightarrow{\lim}\, Q_j$  is \emph{piecewise $\bigwedge_A${\hyp}definable} if it is so as subset $R\subseteq P\times Q$ of the product. We say that $R$ is \emph{piecewise bounded} if the image $R(P_i)\coloneqq \{y\sth \mathrm{there\ is\ }x\in P_i\mathrm{\ with\ }(x,y)\in R\}$ of every piece $P_i$ of $P$ is contained in some piece of $Q$. We say that $R$ is \emph{piecewise proper} if the preimage $R^{-1}(Q_j)\coloneqq\{x\sth \mathrm{there\ is\ }y\in Q_j\mathrm{\ with\ }(x,y)\in R\}$ of every piece $Q_j$ of $Q$ is contained in some piece of $P$. We use this terminology preferably for partial functions. 

Let $P$ be a piecewise hyperdefinable set and $\mu$ an ideal of $\bigwedge_{\omega}${\hyp}definable subsets in $P$. We say that $\mu$ is \emph{$A${\hyp}invariant} if it is invariant by automorphisms fixing $A$ in any elementary extension. We say that a $\bigwedge_{\omega}${\hyp}definable subset is \emph{wide} if it is not in $\mu$. We say that $\mu$ is \emph{locally atomic} if every wide $\bigwedge_B${\hyp}definable subset $V$ of $P$ with $B$ countable contains an element $a\in V$ such that $\tp(a/B)$ is wide. We say that $\mu$ is \emph{compact} if, whenever a countable intersection of $\bigwedge_{\omega}${\hyp}definable subsets of $P$ lies in $\mu$, we have a finite subintersection already contained in $\mu$. We say that $\mu$ is $\bigwedge_A${\hyp}definable if for any partial type $\underline{V}(x,y)$, with $y$ countable, there is a partial type $\dd_\mu x\,\underline{V}(y)$ over $A$ such that $\underline{V}(x,b)\in\mu\Leftrightarrow M\models \dd_\mu x\,\underline{V}(b)$.

Let $R$ be a piecewise bounded $\bigwedge_A${\hyp}definable reflexive binary relation in $P$. We say that $\mu$ is \emph{$R${\hyp}rough $S_1(A)$} if, for any $\bigwedge_{A,b}${\hyp}definable subset $W(b)$ of $P$, where $b$ is countable, we have 
\[R(W(b_0))\cap R(W(b_1))\in\mu \Rightarrow W(b)\in\mu\]
for any $A${\hyp}indiscernible sequence $(b_i)_{i\in\N}$ realizing $\tp(b/A)$. Write $\medium^R_{\mu}(A)$ for the ideal of all $\bigwedge_A${\hyp}definable subsets $V\subseteq P$ such that $\mu_{\mid V}\coloneqq \{W\in\mu\sth W\subseteq V\}$ is $R${\hyp}rough $S_1(A)$. We say that $\mu$ is $S_1(A)$\footnote{Read $S_1$ over $A$. Here we follow the terminology established by the first author in \cite{hrushovski2012stable}. In \cite{rodriguez2020piecewise}, the second author called it \emph{$A${\hyp}medium} following \cite{montenegro2018stabilizers}.} if it is $=${\hyp}rough $S_1(A)$, and write $\medium_\mu(A)$ for $\medium^=_\mu(A)$. As usual, we omit $A$ for $A=\emptyset$. 

A \emph{piecewise $A${\hyp}hyperdefinable group} $G$ is a piecewise $A${\hyp}hyperdefinable set with a group structure whose operations are piecewise bounded $\bigwedge_A${\hyp}definable. A \emph{Lie model} of $G$ is a quotient group homomorphism $\pi\map H\to L=\faktor{H}{K}$ defined on a subgroup $H\leq G$ with $K\trianglelefteq H\leq G$ such that
\begin{enumerate}[label={(\roman*)}, wide]
\item $K$ is $\bigwedge_\omega${\hyp}definable with bounded index in $G$, 
\item $G\setminus H$ and $H$ are piecewise $\bigwedge_\omega${\hyp}definable,  and
\item $L=\faktor{H}{K}$ is a finite dimensional Lie group with the global logic topology. 
\end{enumerate}

Let $\mu$ an ideal of $\bigwedge_{\omega}${\hyp}definable subsets of $G$ invariant under left translations. A \emph{$\mu${\hyp}near{\hyp}subgroup over $A$} of $G$ is a wide $\bigwedge_A${\hyp}definable symmetric subset $X$ containing the identity and generating $G$ such that $\mu_{\mid X^3}$ is a locally atomic $S_1(A)$ ideal. Say $X$ is a near{\hyp}subgroup over $A$ if it is a $\mu${\hyp}near{\hyp}subgroup over $A$ for some ideal $\mu$.

\

We will use the following two theorems from \cite{rodriguez2020piecewise}:

\begin{theo}{\emph{\cite[Theorem 3.28]{rodriguez2020piecewise}}} \label{t:stabilizer theorem}
Let $G$ be a piecewise $0${\hyp}hyperdefinable group, $\mu$ an ideal of $\bigwedge_\omega${\hyp}definable subsets and $X$ a $\mu${\hyp}near{\hyp}subgroup over $\emptyset$. Then, there is a wide $\bigwedge_{\omega}${\hyp}definable normal subgroup of small index $S$ contained in $X^4$ and contained in every $\bigwedge_0${\hyp}definable subgroup of small index. Furthermore, $S=(p\cdot p^{-1})^2$ and $ppp^{-1}=pS=yS$ for any $y\in p$, where $p\subseteq X$ is a wide type over a countable elementary substructure.
\end{theo}

\begin{theo}{\emph{\cite[Theorem 3.31]{rodriguez2020piecewise}}} \label{t:rough lie model}
Let $G^*$ be a $0${\hyp}definable group, $T\leq G^*$ a $\bigwedge_0${\hyp}definable subgroup and $X\subseteq G^*$ a symmetric $\bigwedge_0${\hyp}definable subset. Write $G$ for the subgroup generated by $X$. Assume that $X$ normalises $T$, $\faktor{X}{\kern 0.1em T}$ is a near{\hyp}subgroup of $\faktor{GT}{\kern 0.1em T}$ and $X^n$ is a $T${\hyp}rough approximate subgroup for some $n$. Then, $GT\leq G^*$ has a connected Lie model $\pi\map H\to L=\faktor{H}{K}$ with $T\subseteq K\subseteq X^{2n+4}T$ such that 
\begin{enumerate}[label={\emph{(\arabic*)}}, wide]
\item $H\cap X^{2n}$ is $\bigwedge_{\omega}${\hyp}definable and $T${\hyp}roughly commensurable to $X^n$,
\item $\pi(H\cap X^{2n})$ is a compact neighbourhood of the identity in $L$,
\item $H$ is generated by $H\cap X^{2n+4}T$, and
\item $\pi$ is continuous and proper from the logic topology using countably many parameters.
\end{enumerate}
\end{theo} 

\cref{t:rough lie model} is fully based on Gleason{\hyp}Yamabe's Theorem. We can instead use the variation presented in \cite[Theorem 1.25]{carolino2015structure}, getting some ineffective control on the dimension and the commensurability:

\begin{theo}{\emph{\cite[Theorem 3.32]{rodriguez2020piecewise}}} \label{t:rough lie model carolino}
There are functions $c\map \N\to \N$ and $d\map \N\rightarrow\N$ such that the following holds for any $\aleph_1${\hyp}saturated structure over a countable language:\smallskip

Let $G^*$ be a $0${\hyp}definable group, $T\leq G^*$ a $\bigwedge_0${\hyp}definable subgroup and $X\subseteq G^*$ a symmetric $\bigwedge_0${\hyp}definable subset. Write $G$ for the subgroup generated by $X$. Assume that $X$ normalises $T$, $\faktor{X}{\kern 0.1em T}$ is a near{\hyp}subgroup of $\faktor{GT}{\kern 0.1em T}$ and $X^n$ is a $(k,T)${\hyp}rough approximate subgroup for some $n$ and $k$. Then, $GT\leq G^*$ has a Lie model $\pi\map H\to L=\faktor{H}{K}$ with $T\subseteq K\subseteq X^{12n+4}T$ and $\dim(L)\leq d(k)$ such that
\begin{enumerate}[label={\emph{(\arabic*)}}, wide]
\item $H\cap X^{2n}$ is $\bigwedge_{\omega}${\hyp}definable and $(c(k),T)${\hyp}roughly commensurable to $X^n$,
\item $\pi(H\cap X^{2n})$ is a compact neighbourhood of the identity in $L$,
\item $H$ is generated by $H\cap X^{12n+4}T$, and
\item $\pi$ is continuous and proper from a logic topology using countably many parameters.
\end{enumerate}
\end{theo}
\section{Building ideals}
Fix a countable first{\hyp}order language $\lang$, an $\aleph_1${\hyp}saturated $\lang${\hyp}structure $M$ and a countable set of parameters $A$. Fix an $A${\hyp}definable set $X$. 
\subsection{Making functions definable} Let $f\map X\to \R_{\geq 0}\cup\{\infty\}$ be a function. We say that $f$ is \emph{$A${\hyp}invariant} if $f(a)=f(b)$ when $\tp(a/A)=\tp(b/A)$. From the tradition of continuous logic, $f$ is called \emph{$A${\hyp}definable} when it is continuous using the $A${\hyp}logic topology in $X$. The \emph{standard expansion $M_f$ making $f$ definable} is the one given by adding the predicates $f(x)\leq \alpha$, for $\alpha\in\mathbb{Q}_{\geq 0}$, with the natural interpretations. 

Let $\mathcal{A}$ be a family of definable subsets of $X$ closed under substitution of parameters. A function $\nu\map \mathcal{A}\to \R_{\geq 0}\cup \{\infty\}$ defines a family $\varphi^*\nu(y)$ of functions given by $\varphi^*\nu(a)=\nu(\varphi(M,a))$ for the sets $\varphi(M,a)\in\mathcal{A}$. We say then that $\nu$ is \emph{$A${\hyp}invariant} if each $\varphi^*\nu$ is $A${\hyp}invariant. We say that $\nu$ is \emph{$A${\hyp}definable} if $\varphi^*\nu$ is $A${\hyp}definable for each formula $\varphi$. The \emph{standard expansion $M_{\nu}$ making $\nu$ definable} is the one given by making definable each $\varphi^*\nu$.

For example, a \emph{Keisler measure} on $X$ is a finitely{\hyp}additive probability measure defined on the whole boolean algebra of definable subsets of $X$. Usually, as soon as we add new predicates to the language to make it definable, the Keisler measure is no longer a Keisler measure. Fortunately, every finitely{\hyp}additive probability measure on a subalgebra can be extended to a finitely{\hyp}additive probability measure on the whole algebra. In that case, for a given Keisler measure $\nu_0$, we can recursively build a sequence $(\nu_n)_{n\in\N}$ where $\nu_{n+1}$ is a Keisler measure on $M_{\nu_n}$ extending $\nu_n$. Then, $\nu=\bigcup \nu_n$ is a definable Keisler measure on the expansion $M_{\nu}=\bigcup M_{\nu_n}$ extending the original Keisler measure $\nu_0$ --- of course, properties such as invariance under a group action need not be preserved without special attention.

\subsection{Ultralimits} Let $\{M_i\}_{i\in I}$ be a family of $\lang${\hyp}structures and $f_i\map X(M_i)\rightarrow\R_{\geq 0}\cup\{\infty\}$ functions. Let $\uu$ be an ultrafilter on $I$ and $M=\faktor{\prod M_i}{\uu}$ the respective ultraproduct. We define $f=\lim_{\uu}f_i$ by $f([a_i]_{\uu})=\st [f_i(a_i)]_{\uu}$ for any $a=[a_i]_{\uu}\in X(M)$, where $\st$ is the standard part function.

Suppose that the functions $f_i$ are uniformly definable. In other words, suppose that, for any closed subset $C\subseteq \R\cup\{\pm\infty\}$, there is a common partial type $\Sigma_C$ such that $f^{-1}_i(C)=\Sigma_C(M_i)$ for each $i\in I$. Then, by {\L}o\'{s}'s Theorem, we get $f^{-1}(C)=\Sigma_C(M)$ for each $C\subseteq \R_{\geq 0}\cup\{\infty\}$ closed, so $f$ is definable too.

For functions $\nu_i$ on the boolean algebras of definable subsets of $X(M_i)$, we take $\nu=\lim_{\uu}\nu_i$ defined by $\varphi^*\nu=\lim_{\uu}\varphi^*\nu_i$ for each formula $\varphi$. If all the $\nu_i$ are uniformly definable, $\nu$ is definable. If each $\nu_i$ is subadditive, so is $\nu$. If each $\nu_i$ is a Keisler measure, so is $\nu$. 




\subsection{Liminf of ideals} Let $(\mu_n)_{n\in \N}$ be a sequence of ideals of definable subsets of $X$. Then, $\mu=\lim\inf \mu_n\coloneqq \bigcup_{n_0}\bigcap_{n>n_0}\mu_n$ is an ideal of definable subsets of $X$. Clearly, if each $\mu_n$ is $A${\hyp}invariant, so is $\mu$.

Now, take an $A${\hyp}definable reflexive binary relation $R$ on $X$. Then,
\[\lim\inf\, \medium^R_{\mu_n}(A)\subseteq\medium^R_{\mu}(A).\] 
Indeed, take $Y\subseteq X$ $A${\hyp}definable. Suppose that there is $n_0$ such that $\mu_n$ is $R${\hyp}rough $S_1(A)$ in $Y$ for every $n>n_0$. Pick $Z(b)\subseteq Y$ such that $R(Z(b_0))\cap R(Z(b_1))\in\mu$ where $(b_i)_{i\in\N}$ is indiscernible over $A$ realising $\tp(b/A)$. Then, for some $k>n_0$, $R(Z(b_0))\cap R(Z(b_1))\in\mu_n$ for any $n>k$, so $Z(b)\in \mu_n$ for any $n>k$, concluding $Z(b)\in\mu$. As $Z(b)$ is arbitrary, we conclude that $\mu$ is $R${\hyp}rough $S_1(A)$ in $Y$.

\subsection{Compactifications of ideals of definable sets} Let $\mu$ be an ideal of definable subsets of $X$. We extend $\mu$ to an ideal $\widehat{\mu}$ of $\bigwedge_{\omega}${\hyp}definable subsets of $X$ by compactification as
\[\widehat{\mu}\coloneqq \{V\subseteq X\ {\textstyle{\bigwedge}_{\omega}}\mbox{\hyp}\mathrm{def.}\sth V\subseteq D\mathrm{\ for\ some\ }D\in\mu\}.\] 
If $\mu$ is $A${\hyp}invariant, so is $\widehat{\mu}$. 

By construction $\widehat{\mu}$ is compact. Furthermore, note that this is the unique compact ideal of $\bigwedge_{\omega}${\hyp}definable subsets of $X$ extending $\mu$. 

Note that $\widehat{\mu}$ is also locally atomic. Indeed, given any $\bigwedge_B${\hyp}definable subset $V\notin\widehat{\mu}$, we have $V\cap \bigcap_{D\in\mu_B} D^c\neq\emptyset$ by compactness, where $\mu_B\coloneqq \{D\in\mu\sth D\ B\mbox{\hyp}\mathrm{definable}\}$. Thus, for any $a\in V\cap \bigcap_{D\in \mu_B}D^c$, we get $\tp(a/B)\notin \widehat{\mu}$. 

Let $R$ be an $A${\hyp}definable reflexive binary relation in $X$. Then, clearly,
\[\medium^R_{\mu}(A)=\{D\in\medium^R_{\widehat{\mu}}(A)\sth D\ \mathrm{def.}\}.\]
Indeed, trivially, the left{\hyp}hand side is contained in the right{\hyp}hand side. On the other hand, suppose that $\mu$ is $R${\hyp}rough $S_1(A)$ in $D$, where $D$ is $A${\hyp}definable. Let $W(b)\subseteq D$ be a $\bigwedge_b${\hyp}definable set with $R(W(b_0))\cap R(W(b_1))\in \widehat{\mu}$, where $(b_i)_{i\in \N}$ is an $A${\hyp}indiscernible sequence realising $\tp(b/A)$. By compactness, there is $W_0(b)\subseteq D$ definable with $W(b)\subseteq W_0(b)$ such that $R(W(b_0))\cap R(W(b_1))\subseteq R(W_0(b_0))\cap R(W_0(b_1))\in \mu$. Then, $W(b)\subseteq W_0(b)\in \mu$ as $\mu$ is $R${\hyp}rough $S_1(A)$ in $D$. Since $W(b)$ is arbitrary, we conclude that $\widehat{\mu}$ is $R${\hyp}rough $S_1(A)$ in $D$.

Finally, let $\{R_i\}_{i\in I}$ be a directed family of $\bigwedge_A${\hyp}definable reflexive binary relations on $X$ and $R=\bigcap R_i$. Then, by compactness, 
\[\bigcap \medium^{R_i}_{\widehat{\mu}}(A)=\medium^R_{\widehat{\mu}}(A).\]
Indeed, as $R\subseteq R_i$ for each $i\in I$, $\medium^R_{\widehat{\mu}}(A)\subseteq \bigcap \medium^{R_i}_{\widehat{\mu}}(A)$. On the other hand, suppose that $\widehat{\mu}$ is $R_i${\hyp}rough $S_1(A)$ in $V$ for each $i\in I$ and take $W(b)\subseteq V$ with $b$ countable and assume $R(W(b_0))\cap R(W(b_1))\in \widehat{\mu}$ where $(b_i)_{i\in\N}$ is an $A${\hyp}indiscernible sequence realising $\tp(b/A)$. Then, by compactness, $R_i(W(b_0))\cap R_i(W(b_1))\in \widehat{\mu}$ for some $i\in I$, so $W(b)\in \widehat{\mu}$ as $\widehat{\mu}$ is $R_i${\hyp}rough $S_1(A)$ in $V$, concluding that $V\in\medium^R_{\widehat{\mu}}(A)$. 

\subsection{Mapping ideals} Let $P$ and $Q$ be two piecewise hyperdefinable sets and $\mu$ an ideal of $\bigwedge_{\omega}${\hyp}definable subsets of $P$. Let $f\map P\to Q$ be a piecewise proper $\bigwedge_A${\hyp}definable function. We may then map $\mu$ via $f$ as
\[f_*\mu=\{V\subseteq Q\ {\textstyle{\bigwedge_\omega}}\mbox{\hyp}\mathrm{def.}\sth f^{-1}(V)\in\mu\}.\]
Obviously, $f_*\mu$ is an ideal of $\bigwedge_{\omega}${\hyp}definable subsets of $Q$. It is also clear that $f_*\mu$ is $A${\hyp}invariant as long as $\mu$ is so. 

If $\mu$ is locally atomic, so is $f_*\mu$. Indeed, take a $\bigwedge_B${\hyp}definable subset $V\notin f_*\mu$ with $B$ countable. As $f$ is piecewise proper, $f^{-1}(V)$ is a $\bigwedge_{\omega}${\hyp}definable $\mu${\hyp}wide subset. By local atomicity, there is $a\in f^{-1}(V)$ such that $\tp(a/B)\notin \mu$, so $\tp(f(a)/B)\notin f_*\mu$ as $\tp(a/B)\subseteq f^{-1}(\tp(f(a)/B))$, concluding that $f_*\mu$ is locally atomic too.

Similarly, we have that $f_*\mu$ is compact if $\mu$ is compact. Indeed, take a family $\{W_i\}_{i\in I}$ of $\bigwedge_{\omega}${\hyp}definable subsets of $Q$ such that $\bigcap_{i\in I}W_i\in f_*\mu$ with $I$ countable. Then, $\bigcap_{i\in I}f^{-1}(W_i)=f^{-1}(\bigcap_{i\in I} W_i)\in \mu$, where $\{f^{-1}(W_i)\}_{i\in I}$ is a family of $\bigwedge_{\omega}${\hyp}definable subsets of $P$ as $f$ is piecewise proper. By compactness, there is $I_0\subseteq I$ finite such that $f^{-1}(\bigcap_{i\in I_0}W_i)=\bigcap_{i\in I_0}f^{-1}(W_i)\in\mu$, so $\bigcap_{i\in I_0} W_i\in f_*\mu$, concluding that $f_*\mu$ is compact too. 

Finally, suppose that $f$ is piecewise bounded too. Let $R$ be a piecewise bounded $\bigwedge_A${\hyp}definable reflexive binary relation on $Q$. Then, 
\[\medium^R_{f_*\mu}(A)=\{V\subseteq Q\sth f^{-1}(V)\in\medium^{f^*R}_{\mu}(A)\},\]
where $f^*R\coloneqq \{(x,y)\sth (f(x),f(y))\in R\}$. Indeed, it suffices to note that $f^{-1}(R(Y))=f^*R(f^{-1}(Y))$ for any $Y\subseteq Q$ and $f^*R(X)=f^{-1}(R(f(X)))$ for any subset $X\subseteq P$. Hence, the first identity implies that $\{V\sth f^{-1}(V)\in\medium^{f^*R}_{\mu}(A)\}\subseteq\medium^R_{f_*\mu}(A)$ and the second implies that $\medium^R_{f_*\mu}(A)\subseteq \{V\sth f^{-1}(V)\in\medium^{f^*R}_{\mu}(A)\}$. 

In particular, suppose that $\mu$ is an $E${\hyp}rough $S_1(A)$ ideal of $\bigwedge_{\omega}${\hyp}definable subsets, where $E$ is a $\bigwedge_A${\hyp}definable equivalence relation on $X$. Take the quotient map $\quot\map X\to \faktor{X}{E}$. Then, from the previous discussion, $\quot_*\mu$ is an $S_1(A)$ ideal of $\bigwedge_{\omega}${\hyp}definable subsets of $\faktor{X}{E}$.


\section{Metric Lie Model Theorem}
In this section we prove our main theorem.\medskip

Fix $l\in\N$ and $k=(k_i)_{i\in\N}$ sequence in $\N$. An \emph{$l${\hyp}Lipschitz sequence of growth $k$ in doubling scales} is a sequence $(G_m,X_m,r_{i,m})_{i\leq m\in\N}$ such that
\begin{enumerate}[label={{(\roman*)}}, wide]
\item \label{itm:sequence in doubling scales:metric} $G_m$ is a metric group,
\item \label{itm:sequence in doubling scales:lipschitz} $X_m$ is an $(l,r_{0,m})${\hyp}Lipschitz symmetric subset,
\item \label{itm:sequence in doubling scales:discretisation} $r_m=(r_{i,m})_{i\leq m}$ is a sequence of positive reals with $2r_{i,m}\leq r_{i-1,m}$ and
\[\Nn_{r_{i,m}}(X^9_m)\leq k_i\cdot \Nn_{9r_{i,m}}(X_m)<\infty .\]
\end{enumerate}
\begin{rmk} Let $X$ be a subset of a metric space. For each $r\in \R_{>0}$, write $\Mm_r(X)\coloneqq \frac{\ln\Nn_r(X)}{\ln(1/r)}$. If it exists, the limit $\dim_{\mathrm{MB}}(X)\coloneqq\lim_{r\rightarrow 0^+}\Mm_r(X)$ is called the \emph{Minkowski{\hyp}Bouligand dimension} of $X$. The number $\Mm_r(X)$ is the \emph{$r${\hyp}approximation to the Minkowski{\hyp}Bouligand dimension} of $X$. We can rewrite condition \ref*{itm:sequence in doubling scales:discretisation} above as 
\[\Mm_{r_{i,m}}(X^9_m)-\Mm_{9r_{i,m}}(X_m)\leq\frac{\ln(k_i)}{\ln\left(\sfrac{1}{r_{i,m}}\right)}.\]
Now, to illustrate it, forget for a moment the parameter $m$ and assume $k_i=k$ is constant. Then, this condition is saying that we can find a sequence $(r_i)_{i\in\N}$ converging to $0$ such that 
\[\Mm_{r_i}(X^9)-\Mm_{9r_i}(X)=O\left(\sfrac{1}{\ln\left(\sfrac{1}{r_i}\right)}\right).\]
In particular, that implies that $\dim_{\mathrm{MB}}(X^9)=\dim_{\mathrm{MB}}(X)$ if they exist. Hence, we can understand condition \ref*{itm:sequence in doubling scales:discretisation} above as a comparison between (the approximation sequences of) the Minkowski{\hyp}Bouligand dimensions of $X$ and $X^9$. 

On the other hand, the parameter $m$ is the parameter of the ultraproduct construction. We use it to translate the previous limit comparison into a comparison at finitely many scales.
\end{rmk}

Now, we describe the first{\hyp}order language we are going to use. We consider the language of groups enriched with a predicate for $X_m$ and predicates for $\D_{r_{i,m}}$ for each $i\in\N$ --- where $r_{i,m}=r_{m,m}$ for $i\geq m$. Write $\lang_0$ for this language. We expand $\lang_0$ by making definable, for each $i\in\N$, the normalised covering $r_i${\hyp}discretisation numbers on $X^3$, i.e. the functions 
\[\lambda_i\map Y\mapsto \frac{\Nn^\cov_{r_{i,m}}(Y/X^3_m)}{\Nn^\cov_{r_{i,m}}(X^3_m/X^3_m)}.\]
Write $\lang$ for this language; call it the \emph{associated language} for $(G_m,X_m,r_{i,m})_{i\leq m}$.

Finally, pick a family $\{Z_{i,m}\}_{i,m\in \N}$ of $2r_{i,m}${\hyp}discretisations of $X_m^3$ for each $i,m\in\N$ --- where $Z_{i,m}=Z_{m,m}$ for $i\geq m$ --- and consider the counting probability measures $\nu_{i,m}$ on $Z_{i,m}$ for each $i,m\in\N$. We expand $\lang$ by adding a predicate for each $Z_{i,m}$ and making every $\nu_{i,m}$ definable --- and still making each $\lambda_{i,m}$ definable. Write $\lang'$ for this language; we call it the \emph{auxiliary language} for the sequence $(G_m,X_m,r_{i,m},Z_{i,m})_{i\leq m\in\N}$. Note that $\lang'$ depends on the non{\hyp}canonical choice of $\{Z_{i,m}\}_{i,m\in\N}$.\medskip


We now state the main theorem with all the corresponding details:
\begin{theo}[Metric Lie model] \label{t:main theorem} Let $(G_m,X_m,r_{i,m})_{i\leq m\in\N}$ be an $l${\hyp}Lipschitz sequence of growth $k=(k_i)_{i\in\N}$ in doubling scales. Let $\lang$ be its associated language and, for each $m\in\N$, consider $G_m$ with its natural $\lang${\hyp}structure. Let $G^*=\faktor{\prod G_m}{\uu}$ be a non{\hyp}principal ultraproduct. Write $G\leq G^*$ for the subgroup generated by $X=\faktor{\prod X_m}{\uu}$ and $o_r(1)\coloneqq \bigcap_i \D_{r_i}(1)$ where $\D_{r_i}\coloneqq\faktor{\prod \D_{r_{i,m}}}{\uu}$. Then:
\begin{enumerate}[label={\emph{(\arabic*)}}, wide]
\item $\faktor{X}{o_r(1)}$ is a near{\hyp}subgroup (with respect to $\lang$).
\item If $k$ is constant, $X^2$ is a $(k,o_r(1))${\hyp}rough approximate subgroup.
\end{enumerate}
In particular, if $k$ is constant, $G\cdot o_r(1)\leq G^*$ has a connected Lie model $\pi\map H\to L=\faktor{H}{K}$ with $o_r(1)\trianglelefteq K\subseteq X^8\cdot o_r(1)$ such that
\begin{enumerate}[label={\emph{(\alph*)}}, wide]
\item $H\cap X^4$ is a $\bigwedge_{\omega}${\hyp}definable subset $o_r(1)${\hyp}roughly commensurable to $X^2$, 
\item $\pi(H\cap X^4)$ is a compact neighbourhood of the identity in $L$, and
\item $H\cap X^8\cdot o_r(1)$ generates $H$.
\end{enumerate}
\end{theo}
\begin{rmk} Let $V$ be compact and $U$ open with $V\subseteq U\subseteq L$. By definition of the Lie model, we have that $\pi^{-1}(V)$ is $\bigwedge${\hyp}definable, $\pi^{-1}(U)^c$ is piecewise $\bigwedge${\hyp}definable and $\pi^{-1}(V) \subseteq \pi^{-1}(U)$. Hence, there is a definable subset $D$ such that $\pi^{-1}(V)\subseteq D\cdot o_r(1)\subseteq \pi^{-1}(U)$. As $L$ is second countable and locally compact, we conclude that $\pi^{-1}(V)$ is $\bigwedge_\omega${\hyp}definable and $\pi^{-1}(U)=U_0\cdot o_r(1)$ for some $\bigvee_\omega${\hyp}definable subset $U_0$. 
\end{rmk}
\begin{proof} First of all, note that the ultraproduct is $\aleph_1${\hyp}saturated \cite[Exercise 5.2.3]{tent2012course}, so the piecewise $0${\hyp}hyperdefinable set $\faktor{G\cdot o_r(1)}{o_r(1)}$ is well defined. By \cref{l:subgroup of infinitesimals}, $X$ normalises $o_r(1)$, so $\faktor{G\cdot o_r(1)}{o_r(1)}$ is a piecewise $0${\hyp}hyperdefinable group. 
\begin{enumerate}[label={\rm{(\arabic*)}}, itemsep=5pt, topsep=5pt, wide]
\item We start by noting that, for $i\in\N$ and $Y\subseteq X^3$ definable in $\lang'$,
\[\lambda_i(Y)=0\Leftrightarrow \nu_i(\D_{r_i}(Y)\cap Z_i)=0.\]
Write $Y_m\coloneqq Y(G_m)$. By \cref{l:subadditivity and monotonicity of discretisation numbers,l:equivalence discretisation numbers,l:lemma 1 application to metric groups}, we have
\[\Nn^\cov_{2r_{i,m}}(Y_m/X^3_m)\leq \Nn_{2r_{i,m}}(Y_m) \leq \Nn^\cov_{r_{i,m}}(Y_m/X_m^3) \leq k_i\cdot \Nn^\cov_{2r_{i,m}}(Y_m/X^3_m).\]
On the other hand, as $\Nn_{r_{i,m}}(X^9_m)\leq k_i\cdot\Nn_{9r_{i,m}}(X_m)$, by \cref{l:subadditivity and monotonicity of discretisation numbers,l:equivalence discretisation numbers}, we get
\[\frac{1}{\Nn_{r_{i,m}}(X^3_m)}\leq \frac{1}{\Nn^\cov_{r_{i,m}}(X^3_m/X^3_m)}\leq k_i\cdot \frac{1}{\Nn_{r_{i,m}}(X^3_m)}.\]
By \cref{l:lemma 2 application to metric groups}, we conclude 
\[\frac{1}{k_i}\cdot \lambda_{i,m}(Y_m)\leq \nu_{i,m}(Z_{i,m}\cap \D_{r_{i,m}}(Y_m))\leq k_i\cdot \lambda_{i,m}(Y_m).\]
Therefore, by \L{o}\'{s}'s Theorem, we conclude that 
\[\lambda_i(Y)=0\Leftrightarrow \nu_i(\D_{r_i}(Y)\cap Z_i)=0.\]

For each $i\in\N$, let $\mu'_i$ be the ideal of definable subsets of $X^3$ in $\lang'$ given by $Y\in \mu'_i\Leftrightarrow \lambda_i(Y)=0$. Let $\mu_i$ be its $\lang${\hyp}reduct. \smallskip

Write $\D{}^{X^3}_{r_i}\coloneqq \D_{r_i}\cap (X^3\times X^3)=\{(x,y)\in X^3\times X^3\sth \dd(x,y)\leq r_i\}$.
\begin{claim} $\mu_i$ is an invariant under left translations $\D{}_{r_i}^{X^3}${\hyp}rough $S_1$ ideal of definable subsets of $X^3$ with $X$ wide, with respect to the language $\lang$. Similarly, $\mu_i'$ is an invariant under left translations $\D{}_{r_i}^{X^3}${\hyp}rough $S_1$ ideal of definable subsets of $X^3$ with $X$ wide, with respect to the language $\lang'$.
\begin{proof} $\mu_i$ and $\mu'_i$ are clearly ideals by subadditivity of the discretisation numbers, \cref{l:subadditivity and monotonicity of discretisation numbers}. Now, $\mu_i$ and $\mu'_i$ are $0${\hyp}invariant by definability of $\lambda_i$. As $\lambda_i$ is invariant under left translations, it is clear that $\mu_i$ and $\mu'_i$ are invariant under left translations. By \cref{l:equivalence discretisation numbers,l:subadditivity and monotonicity of discretisation numbers}, we get that $\Nn^\cov_{r_{i,m}}(X^3_m/X^3_m)\leq k_i\cdot \Nn_{9r_{i,m}}(X_m)\leq k_i\cdot \Nn^\cov_{r_{i,m}}(X_m/X^3_m)$. Therefore, $\lambda_i(X)\geq \sfrac{1}{k_i}$, so $X$ is wide. 

Finally, with respect to $\lang'$, let $Y(b)\subseteq X^3$ be $b${\hyp}definable and $(b_i)_{i\in\N}$ a $0${\hyp}indiscernible sequence starting at $b$. Suppose 
\[X^3\cap \D_{r_i}(Y(b_0))\cap \D_{r_i}(Y(b_1))\in\mu'_i.\]
Then, 
\[\nu_i(Z_i\cap \D_{r_i}(X^3\cap \D_{r_i}(Y(b_0))\cap \D_{r_i}(Y(b_1))))=0,\]
so, in particular, $\nu_i(Z_i\cap \D_{r_i}(Y(b_0))\cap \D_{r_i}(Y(b_1)))=0$. By invariance of $\nu_i$, we have $\nu_i(Z_i\cap \D_{r_i}(Y(b_j)))=\nu_i(Z_i\cap \D_{r_i}(Y(b_0)))$. As $\nu_i$ is a Keisler measure,
\[1=\nu_i(Z_i)\geq \sum^{\infty}_{j=0} \nu_i(Z_i\cap \D_{r_i}(Y(b_j)))=\sum^{\infty}_{j=0} \nu_i(Z_i\cap \D_{r_i}(Y(b_0))).\]
Therefore, $\nu_i(Z_i\cap\D_{r_i}(Y(b_0)))=0$, so $Y(b_0)\in\mu'_i$. As $Y$ is arbitrary, we conclude that $\mu'_i$ is a $\D{}^{X^3}_{r_i}${\hyp}rough $S_1$ ideal.

Now, with respect to $\lang$, let $Y(b)\subseteq X^3$ be $b${\hyp}definable and $(b_i)_{i\in\N}$ a $0${\hyp}indiscernible sequence starting at $b$. Suppose 
\[X^3\cap \D_{r_i}(Y(b_0))\cap \D_{r_i}(Y(b_1))\in\mu_i.\] 
As $\mu_i$ is $0${\hyp}invariant, we have $X^3\cap \D_{r_i}(Y(b_j))\cap \D_{r_i}(Y(b_{j'}))\in\mu_i$ for $j<j'$. By the Standard Lemma \cite[Theorem 5.1.5]{tent2012course}, we can find a $0${\hyp}indiscernible sequence $(b'_i)_{i\in\N}$ starting at $b$ with respect to $\lang'$ such that
\[X^3\cap \D_{r_i}(Y(b'_0))\cap \D_{r_i}(Y(b'_1))\in\mu_i\subseteq \mu'_i.\]
Then, by the rough $S_1$ property of $\mu'_i$, we conclude that $Y(b_0)\in\mu'_i$, so $Y(b_0)\in\mu_i$. As $Y$ is arbitrary, we conclude that $\mu_i$ is a $\D{}_{r_i}^{X^3}${\hyp}rough $S_1$ ideal.\claimqed
\end{proof}
\end{claim}
From now on we only work in $\lang$. Let $\mu_{\infty}=\bigcup^{\infty}_{n=1}\bigcap^{\infty}_{i=n}\mu_i=\liminf\mu_i$. Then, $\mu_{\infty}$ is an invariant under left translations $\D{}^{X^3}_{r_i}${\hyp}rough $S_1$ ideal of definable subsets of $X^3$ with $X\notin\mu_{\infty}$ for every $i\in\N$. Consider 
\[\widehat{\mu}_{\infty}=\left\{W\ {\textstyle{\bigwedge_{\omega}}}\mbox{\hyp}\mathrm{definable}\sth W\subseteq Y\in \mu_\infty \mathrm{\ for\ some\ }Y\mathrm{\ definable}\right\}.\]
Then, $\widehat{\mu}_{\infty}$ is an invariant under left translations compact locally atomic $o^{X^3}_r${\hyp}rough $S_1$ ideal of $\bigwedge_{\omega}${\hyp}definable subsets of $X^3$ with $X\notin\widehat{\mu}_{\infty}$, where $o^{X^3}_r\coloneqq \bigcap_i \D{}^{X^3}_{r_i}=o_r\cap (X^3\times X^3)=\{(x,y)\in X^3\times X^3\sth x^{-1}y\in o_r(1)\}$.

Consider the canonical quotient map $\quot\map X^3\to \faktor{X^3}{o_r^{X^3}}\cong \faktor{X^3}{o_r(1)}$. Clearly, $\quot$ is a piecewise bounded and proper $\bigwedge_0${\hyp}definable function. Define 
\[\mu\coloneqq \quot_*\widehat{\mu}_{\infty}=\{Y\sth \quot^{-1}(Y)\in\widehat{\mu}_\infty\}.\]
Hence, $\mu$ is an invariant under left translations compact locally atomic $S_1$ ideal of $\bigwedge_{\omega}${\hyp}definable subsets of $\faktor{X^3}{o_r(1)}$ with $\faktor{X}{o_r(1)}$ wide. In other words, $\faktor{X}{o_r(1)}$ is a near{\hyp}subgroup.
\item Using \cref{l:subadditivity and monotonicity of discretisation numbers}, 
\[\Nn_{r_{i,m}}(X^5_m)\leq \Nn_{r_{i,m}}(X^9_m)\leq k_i\cdot \Nn_{9r_{i,m}}(X_m)\leq k\cdot \Nn_{r_{i,m}}(X_m).\]
By \cref{l:metric approximate subgroups and discretisation numbers}, $X^2_m$ is a $(k_i,2lr_{i,m})${\hyp}metric approximate subgroup for each $i\leq m\in\N$. Now, assuming that $k\coloneqq k_i$ is constant, by \L{o}\'{s}'s Theorem, we get that $X^2$ is a $(k,2lr_i)${\hyp}metric approximate subgroup for each $i\in\N$. Thus, we conclude that $X^2$ is a $(k,o_r(1))${\hyp}rough approximate subgroup. 
\end{enumerate}
Finally, applying \cref{t:rough lie model}, we get the desired Lie model.
\end{proof}

\begin{rmk} \begin{enumerate}[label={\rm{(\arabic*)}}, wide]
\item[\hspace{-1.4em}\setcounter{enumi}{1}\theenumi] If $\sfrac{r_i}{2}\leq \sfrac{9r_{i+1}}{4}$, then $\mu_{i+1}\subseteq \mu_i$. In particular, if $r_i\leq \sfrac{9r_{i+1}}{2}$ for every $i\in\N$, then $\mu_{\infty}=\bigcap \mu_i$. In this case, as each $\mu_i$ is $\bigwedge_0${\hyp}definable, we conclude that $\mu_{\infty}$ is $\bigwedge_0${\hyp}definable.
\item If we get that $X$ is an $o_r(1)${\hyp}rough approximate subgroup, we can reduce the powers in \cref{t:main theorem}. Explicitly, we can get that $K\subseteq X^6\cdot o_r(1)$, $H\cap X^2$ and $X$ are $o_r(1)${\hyp}roughly commensurable and that $H$ is generated by $H\cap X^6\cdot o_r(1)$.
\item $Y=H\cap X^8$ is a definable subset $o_r(1)${\hyp}roughly commensurable to $X^2$ such that $Y\cdot o_r(1)$ generates $H$.
\item Without significant loss of generality, we may focus only on the case when $r_{0,m}\leq \diam(X_m)$ for all $m\in\N$. Indeed, after the ultraproduct, up to taking a subsequence, this assumption is equivalent to imposing that there is $i\in \N$ such that $r_i\leq \diam(X)$. Now, this only excludes the case $X\subseteq o_r(1)$, for which \cref{t:main theorem} trivially holds. 
\end{enumerate}
\end{rmk}

\

Alternatively, using \cref{t:rough lie model carolino} rather than \cref{t:rough lie model}, we get the following variation:

\begin{coro}[Metric Lie model, version 2] \label{c:metric lie model carolino} There are functions $c\map \N\to \N$ and $d\map \N\to \N$ such that the following holds:\smallskip

Let $(G_m,X_m,r_{i,m})_{i\leq m\in\N}$ be an $l${\hyp}Lipschitz sequence of constant growth $k$ in doubling scales. Let $\lang$ be its associated language and, for each $m\in\N$, consider $G_m$ with its natural $\lang${\hyp}structure. Let $G^*=\faktor{\prod G_m}{\uu}$ be a non{\hyp}principal ultraproduct. Write $G\leq G^*$ for the subgroup generated by $X=\faktor{\prod X_m}{\uu}$. Then, $G\cdot o_r(1)\leq G^*$ has a Lie model $\pi\map H\to L=\faktor{H}{K}$ with $o_r(1)\trianglelefteq K\subseteq X^{28}\cdot o_r(1)$ and $\dim(L)\leq d(k)$ such that 
\begin{enumerate}[label={\emph{(\alph*)}}, wide]
\item $H\cap X^4$ is $\bigwedge_{\omega}${\hyp}definable and $(c(k),o_r(1))${\hyp}roughly commensurable to $X^2$, 
\item $\pi(H\cap X^4)$ is a compact neighbourhood of the identity in $L$, and
\item $H\cap X^{28}\cdot o_r(1)$ generates $H$.
\end{enumerate}
\end{coro}

\section{Applications}
Now, we give some applications of our main \cref{t:main theorem}. Using \cref{l:lemma sequence of growth in doubling scale}, the following corollaries give us \cref{c:corollary 1,c:corollary 2,c:corollary 3} respectively. 

\begin{coro} \label{c:corollary 1 general} Fix $l$, $n$, $s$ and $k=(k_i)_{i\in\N}$ in $\N$. There is $m\coloneqq m(k,l,n,s)$ such that the following holds:\smallskip

Let $G$ be a metric group, $X$ an $(l,r_0)${\hyp}Lipschitz symmetric subset and $r_0,\ldots,$ $r_m$ with $2r_i\leq r_{i-1}$ such that 
\[\Nn_{r_i}(X^9)\leq k_i\cdot \Nn_{9r_i}(X)<\infty.\]
Then, there is $I\subseteq \{0,\ldots,m-1\}$ with $|I|=n$ and a symmetric subset $Y$ with $x^{-1}Y^nx\subseteq \D_{r_s}(X^4)$ for any $x\in X^n$ such that 
\[\Nn_{r_i}(Y)\geq \frac{1}{m} \Nn_{r_i}(X) \mathrm{\ for\ all\ }i\in I.\]
\end{coro}
\begin{proof}
Aiming a contradiction, suppose otherwise. Thus, we get an $l${\hyp}Lipschitz sequence of growth $k$ in doubling scales $(G_m,X_m,r_{i,m})_{i\leq m\in\N}$ of counterexamples. Let $(G^*,X,\ldots)$ be a non{\hyp}principal ultraproduct as in \cref{t:main theorem}. Write $G$ for the $\bigvee_0${\hyp}definable subgroup generated by $X$ and consider $\faktor{G\cdot o_r(1)}{o_r(1)}$.

We have that $\faktor{X}{o_r(1)}$ is a near{\hyp}subgroup by \cref{t:main theorem}. Moreover, recall that, in \cref{t:main theorem}, the ideal $\mu$ of $\bigwedge_\omega${\hyp}definable subsets of $\faktor{X^3}{o_r(1)}$ is defined as $\mu=f_*\widehat{\mu}_{\infty}$, where $f=\quot_{\mid X^3}\map X^3\to \faktor{X^3}{o_r(1)}$ is the canonical projection and $\widehat{\mu}_{\infty}$ is the compactification of $\mu_{\infty}=\lim\inf\, \mu_i$ with \[Y\in \mu_i\Leftrightarrow \frac{\Nn^\cov_{r_i}(Y/X^3)}{\Nn^\cov_{r_i}(X^3/X^3)}=0.\] 

By the Stabilizer \cref{t:stabilizer theorem}, we get a $\bigwedge_\omega${\hyp}definable normal subgroup $S'\leq \faktor{G}{o_r(1)}$ of bounded index contained in $\faktor{X^4}{o_r(1)}$ such that $S'=(pp^{-1})^2$ for some wide type $p$ in $\faktor{X}{o_r(1)}$ and $ppp^{-1}=pS'=aS'$ for any $a\in p$.

Note that $\quot\map G\cdot o_r(1)\to \faktor{G\cdot o_r(1)}{o_r(1)}$ is a piecewise bounded and proper $\bigwedge_0${\hyp}definable homomorphism. Take $S=\quot^{-1}(S')$. Then, $S\trianglelefteq G\cdot o_r(1)$ is $\bigwedge_\omega${\hyp}definable with $x^{-1}S^nx=S\subseteq X^4\cdot o_r(1)$ for any $x\in X^n$ and $o_r(1)\leq S$. Now, by compactness, we find a definable symmetric set $Y$ containing $S$ such that $x^{-1}Y^nx\subseteq \D_{r_s}(X^4)$ for any $x\in X^n$. 

Pick $a\in p$, so we have $a\cdot S'\subseteq \faktor{X^3}{o_r(1)}$. Write $S'_a\coloneqq a^{-1}S'\cap \faktor{X^3}{o_r(1)}$, $S_a\coloneqq a^{-1}S\cap X^3=\quot^{-1}_{\mid X^3}(S'_a)$ and $Y_a\coloneqq a^{-1}Y\cap X^3$. As $S'$ is wide, we know that $S'_a$ is wide. Hence, $S_a\notin \widehat{\mu}_{\infty}$ and so $Y_a\notin \mu_{\infty}$. Then, $Y_a\notin \mu_i$ for infinitely many $i\in\N$. Choose $n<i_1<i_2<\cdots<i_n$ such that $Y_a\notin \mu_{i_t}$ for each $t\in\{1,\ldots,n\}$, and set $I=\{i_1,\ldots,i_n\}$. Hence, for each $i\in I$, there is $m_i\in\N$ such that 
\[\frac{\Nn^\cov_{r_i}(Y/X^3)}{\Nn^\cov_{r_i}(X^3/X^3)}\geq \frac{\Nn^\cov_{r_i}(Y_a/X^3)}{\Nn^\cov_{r_i}(X^3/X^3)}\geq \frac{1}{m_i}.\]
Using also \cref{l:subadditivity and monotonicity of discretisation numbers,l:equivalence discretisation numbers}, we get 
\[\frac{\Nn_{r_i}(Y)}{\Nn_{r_i}(X)}\geq \frac{1}{k_i}\frac{\Nn^\cov_{r_i}(Y/X^3)}{\Nn^\cov_{r_i}(X^3/X^3)}\geq \frac{1}{k_im_i}.\] 
Taking $m_0=\max\{k_im_i\sth i\in I\}$, we conclude $\sfrac{\Nn_{r_i}(Y)}{\Nn_{r_i}(X)}\geq \frac{1}{m_0}$ for each $i\in I$.

By {\L}o\'{s}'s Theorem, we have some $m>\max\{m_0,i_1,\ldots,i_n,s\}$ such that $I\subseteq \{0,\ldots,m-1\}$ with $|I|=n$ and, for $Y_m=Y(G_m)$, 
\[\begin{array}{ll} 
\mathrm{(1)}& 1\in Y_m=Y_m^{-1},\\
\mathrm{(2)} & x^{-1}Y_m^{n}x\subseteq \D_{r_s}(X_m^4)\mathrm{\ for\ all\ }x\in X^n_m, \mathrm{\ and}\\
\mathrm{(3)}& \Nn_{r_{i,m}}(Y_m)\geq \frac{1}{m} \Nn_{r_{i,m}}(X_m)\ \mathrm{for\ any}\ i\in I;\end{array}\]
contradicting the assumption that $(G_m,X_m,(r_{i,m})_{i\leq m})$ is a counterexample. 
\end{proof}

For the following corollary, we will use a slightly improved version of \cite[Corollary 4.18]{hrushovski2012stable} which already appears in \cite[Corollary 5.6]{dries2015approximate}. For convenience, we provide here the full statement: 
\begin{lem} \label{l:lemma corollary 2} Fix $k,N\in\N$. Then, there exists $c=c(k,N)$ with the following property: \smallskip

Let $G$ be a group and $X$ a finite $k${\hyp}approximate subgroup such that $x^N=1$ for any $x\in X^2$. Then, there is a subgroup $S$ of $G$ with $S\subseteq X^4$ such that $S$ and $X$ are $c${\hyp}commensurable.
\end{lem}
\begin{proof} Aiming a contradiction, suppose otherwise. Then, for each $c\in \N$ we have a counterexample $(G_c,X_c)_{c\in\N}$. Take a non{\hyp}principal ultraproduct $(G,X)$. Using the ultralimit of the counting measures (and taking the language making it definable), it follows that $X$ is a definable near{\hyp}subgroup $k${\hyp}approximate subgroup. By the Lie model Theorem \cite[Theorem 4.2]{hrushovski2012stable}, we find a connected Lie model $\pi\map H\to L=\faktor{H}{K}$ where $K\subseteq X^4$, $\pi(H\cap X^2)$ is a compact neighbourhood of the identity in $L$ and $H\cap X^4$ generates $H$ and is $e${\hyp}commensurable to $X$ for some $e$. 

Now, for every $g\in\pi(H\cap X^2)$, $g^N=1$. As $L$ is a Lie group, it has a neighbourhood $U_0$ of the identity such that $g^2=h^2$ if and only if $g=h$. Take a neighbourhood of the identity $U$ in $L$ such that $U^N\subseteq U_0$. Therefore, for any $g,h\in U$, $g^N=h^N$ if and only if $g=h$. Hence, $U\cap \pi(H\cap X^4)=\{1\}$, concluding that $L$ is discrete. As it is connected, then $L$ is trivial, so $S\coloneqq K=H=H\cap X^4$ is a definable subgroup contained in $X^4$ and $e${\hyp}commensurable to $X$. By {\L}o\'{s}'s Theorem, we conclude that there is some $c\geq e$ such that $S_e$ is a subgroup of $G_e$ contained in $X^4$ and $e${\hyp}commensurable to $X_e$, contradicting that $(G_e,X_e)$ is a counterexample for $e$.
\end{proof}
\begin{coro} \label{c:corollary 2 general}
Fix $k,l,N,n,s\in \N$. There are $m\coloneqq m(k,l,N,n,s)\in\N$ and $c=c(k,N)\in\N$ such that the following holds:\smallskip

Let $G$ be a metric group, $X$ a $(l,r_0)${\hyp}Lipschitz symmetric subset and $r_0,\ldots,$ $r_m$ with $2r_i\leq r_{i-1}$ such that $\dd(g^N,1)\leq r_m$ for all $g\in X^8$ and 
\[\Nn_{r_i}(X^9)\leq k\cdot \Nn_{9r_i}(X)<\infty.\]
Then, there is a symmetric subset $Y\subseteq X^{16}$ such that $Y$ and $X^2$ are $(c,r_s)${\hyp}commensurable and $Y^n\subseteq \D_{r_s}(Y)$.
\end{coro}
\begin{proof} Aiming a contradiction, suppose otherwise. Thus, we get an $l${\hyp}Lipschitz sequence of growth $k$ in doubling scales $(G_m,X_m,r_{i,m})_{i\leq m\in\N}$ of counterexamples. Take a non{\hyp}principal ultraproduct $(G^*,X,\ldots)$ as in \cref{c:metric lie model carolino}. Write $G$ for the $\bigvee_0${\hyp}definable subgroup generated by $X$. By \cref{c:metric lie model carolino}, we have a Lie model $\pi\map H\to L\coloneqq\faktor{H}{K}$ of $G\cdot o_r(1)$ with $o_r(1)\subseteq K\subseteq X^{28}\cdot o_r(1)$ such that $\pi(H\cap X^4)$ is a neighbourhood of the identity in $L$ and $H\cap X^4$ is $(c_0(k),o_r(1))${\hyp}roughly commensurable to $X^2$, where $c_0(k)$ only depends on $k$. 

Now, since $\dd(g^N,1)<r_{m,m}<r_{i,m}$ for every $g\in X^8_m$ and $i\leq m\in \N$, it follows by {\L}o\'{s}' Theorem that $g^N\in o_r(1)$ for every $g\in X^8$. Then, $\pi(g)^N=1$ for every $g\in H\cap X^4$. Now, $\pi(H\cap X^4)$ is a neighbourhood of the identity of $L$. As $L$ is a Lie group, it has a neighbourhood $U_0$ of the identity such that $g^2=h^2$ if and only if $g=h$. Take a neighbourhood of the identity $U$ in $L$ such that $U^N\subseteq U_0$. Therefore, for any $g,h\in U$, $g^N=h^N$ if and only if $g=h$. Hence, $U\cap \pi(H\cap X^4)=\{1\}$, concluding that $L$ is discrete. As $\pi(H\cap X^4)$ is compact, we have that it is finite. Note that $\pi(H\cap X^4)$ is a $k^3${\hyp}approximate subgroup by \cite[Lemma 2.3]{machado2023closed} and $g^N=1$ for any $g\in\pi(H\cap X^4)^2$. Applying \cref{l:lemma corollary 2}, it follows that there is a finite subgroup $S_0\subseteq \pi(H\cap X^4)^4$ $c_1(k,N)${\hyp}commensurable to $\pi(H\cap X^4)$, where $c_1(k,N)$ is a constant that only depends on $k$ and $N$. Since $\pi$ is a group homomorphism, we get that $S\coloneqq \pi^{-1}(S_0)\subseteq X^{16}\cdot o_r(1)$ is a subgroup of $G\cdot o_r(1)$ $c${\hyp}commensurable to $X^2\cdot o_r(1)$, where $c\coloneqq c_0(k)c_1(k,N)$ only depends on $k$ and $N$. As $\pi$ is continuous from the logic topology, $S$ is open and compact, concluding that it is a relative definable subgroup of $G\cdot o_r(1)$ contained in $X^{16}\cdot o_r(1)$. 

Take $Y=S\cap X^{16}$. Hence, $Y\subseteq X^{16}$ is definable with $Y^n\subseteq H=Y\cdot o_r(1)$ and $(c,o_r(1))${\hyp}roughly commensurable to $X^2$. We have then $Y^n\subseteq \D_{r_s}(Y)$ and $Y$ is $(c,r_s)${\hyp}commensurable to $X^2$. By {\L}o\'{s}' Theorem, we conclude that there is some $m\in\N$ such that $Y_m\subseteq X^{16}_m$ is $(c,r_{s,m})${\hyp}commensurable to $X^2_m$ with $Y^n_m\subseteq \D_{r_{s,m}}(Y_m)$, contradicting the assumption that $(G_m,X_m,(r_{i,m})_{i\leq m})$ is a counterexample.  
\end{proof}

For the last corollary, we will use the following basic fact about Lie groups. A similar fact was already used and explained in \cite[Corollary 4.13]{hrushovski2012stable}. 
\begin{lem} \label{l:lemma corollary 3} Let $L$ be a Lie group and $U$ a compact neighbourhood of the identity. There are a sequence $(U_n)_{n\in\N}$ of symmetric compact neighbourhoods of the identity with $U_{n+1}\subseteq U^\circ_n$ and $U_0\subseteq U^{\circ}$ satisfying the following properties:
\[\begin{array}{ll} 
\mathrm{(1)}&U_n \mathrm{\ is\ covered\ by\ }17^{\dim(L)}\mathrm{\ translates\ of\ }U_{n+1}\mathrm{\ for\ each\ }n\in\N.\\
\mathrm{(2)}& U_{n+1}^2\subseteq U_n\mathrm{\ for\ each\ }n\in\N.\\
\mathrm{(3)}& x^{-1}U_{n+1}x\subseteq U_n\mathrm{\ for\ each\ }x\in U_0\mathrm{\ and\ }n\in\N.\\
\mathrm{(4)}& [U_{n_1},U_{n_2}]\subseteq U_{n}\mathrm{\ for\ each\ }n,n_1,n_2\in\N\mathrm{\ with\ }n\leq n_1+n_2.\\
\mathrm{(5)}& \mathrm{If\ }x^2=y^2\mathrm{\ with\ }x,y\in U_0\mathrm{,\ then\ }x=y.\\
\mathrm{(6)}& U_{n+4}=\{x\in U_0\sth x^{17}\in U_n\}\mathrm{\ for\ each\ }n\in\N.
\end{array}\]
\end{lem}
\begin{proof} Let $U'$ be a small neighbourhood of the identity in $U$ where the exponential map is a diffeomorphism. Take $\widetilde{U}$ symmetric neighbourhood of the identity with $\widetilde{U}^2\subseteq U'$ and take $B=\log(\widetilde{U})$ be the corresponding neighbourhood of the identity of the Lie algebra. Fix some inner product in the Lie algebra. 

Consider the operation $x \ast y\coloneqq \log(\exp(x)\exp(y))$ defined in the Lie algebra for $x,y\in B$. Note that $x\ast 0=x=0\ast x$, $x\ast (-x)=(-x)\ast x=0$  and $(x\ast y)\ast z=x\ast (y\ast z)$ when defined. Also, note that $\ast$ is smooth, so there is $\varepsilon_0>0$ and $C_0>0$ such that
\[\|x \ast y-(x+y)\|\leq C_0\|x\|\|y\|,\]
for $\|x\|,\|y\|<\varepsilon_0$. By the triangular inequality, one can find $\varepsilon_1>0$ and $C_1>0$ such that, for $\|x\|,\|y\|<\varepsilon_1$, 
\[\|x\ast y\ast (-x)\ast (-y)\|\leq C_1\|x\|\|y\|\mathrm{\ and\ }\|x\ast y\ast (-x)-y\|\leq C_1\|x\|\|y\|.\] 

Take $\varepsilon=\min\{\varepsilon_1,\frac{1}{2C_1},\frac{1}{17 C_0}\}$ and set $B_n=\mathbb{D}(0,\sqrt[4]{17}^{-n-1}\varepsilon)$ and $U_n=\exp(B_n)$. Then, we get that $(U_n)_{n\in\N}$ satisfies the desired conditions: 
\begin{enumerate}[label={\rm{(\arabic*)}}, wide]
\item Note that $\Vol(B_{n+1})=\sqrt[4]{17}^{-d}\cdot\Vol(B_n)$ where $d=\dim(L)$. Take $Z\subseteq B_n$ maximal such that $\{z+B_{n+1}\}_{z\in Z}$ is a disjoint family. Then, $B_n\subseteq Z+B_{n+1}$ and $|Z|\Vol(B_{n+1})=\Vol(Z+B_{n+1})\leq \Vol(B_{n-1})$, concluding that $|Z|\leq \sqrt{17}^{d}$ where $d=\dim(L)$. Take $Z\subseteq B_{n-1}$ such that $B_n\subseteq Z+B_{n+2}$ with $|Z|\leq 17^{d}$, where $d=\dim(L)$. Then, for every element $x\in B_n$, there is $z\in Z$ such that $x-z\in B_{n+1}$. Hence, $\|(-z)\ast x\|\leq \sqrt[4]{17}^{-n-2}\varepsilon+C_0\sqrt[4]{17}^{-2n-1}\varepsilon^2\leq \sqrt[4]{17}^{-n-1}\cdot\varepsilon$. Thus, $\exp(Z)$ satisfies $U_n\subseteq \exp(Z) U_{n+1}$ with $|\exp(Z)|\leq 17^d$ with $d=\dim(L)$.
\item Given $a,b\in U_n$, then $a=\exp(x)$, $b=\exp(y)$ where $\|x\|,\|y\|\leq \sqrt[4]{17}^{-n-1}\varepsilon$. Thus, $\|x\ast y\|\leq \|x+y\|+C_0\|x\|\|y\|\leq (2\cdot \sqrt[4]{17}^{-n-1}+\frac{1}{17} \sqrt[4]{17}^{-2n-2}) \varepsilon\leq (\frac{2}{\sqrt[4]{17}}+\frac{\sqrt[4]{17}^{-n}}{17\sqrt{17}}) \sqrt[4]{17}^{-n}\varepsilon\leq (\frac{2}{\sqrt[4]{17}}+\frac{1}{17\sqrt{17}})\sqrt[4]{17}^{-n}\varepsilon\leq \sqrt[4]{17}^{-n}\varepsilon$. Hence, $x\ast y\in B_{n-1}$, concluding that $ab=\exp(x\ast y)\in U_{n-1}$.  
\item Given $a\in U_n$ and $g\in U_0$, we get $a=\exp(y)$ and $g=\exp(x)$ with $\|y\|\leq \sqrt{5^{-n-1}}\varepsilon$ and $\|x\|\leq \varepsilon$. Hence, $\|x\ast y\ast(-x)\|\leq \|y\|+C_1\|x\|\|y\|\leq \sqrt[4]{17}^{-n-1}\cdot (1+\frac{1}{2})\varepsilon\leq\sqrt[4]{17}^{-n}\varepsilon$, so $gag^{-1}=\exp(x\ast y\ast (-x))\in U_{n-1}$.   
\item  Given $a\in U_{n_1}$ and $b\in U_{n_2}$ with $n\leq n_1+n_2$, we get $a=\exp(x)$ and $b=\exp(y)$ with $\|x\|\leq \sqrt[4]{17}^{-n_1-1}\varepsilon$ and $\|y\|\leq \sqrt[4]{17}^{-n_2-1}\varepsilon$. Thus, $\|x\ast y\ast (-x)\ast(-y)\| \leq C_1\|x\|\|y\|\leq \frac{\sqrt[4]{17}^{-1}}{2} \sqrt[4]{17}^{-(n_1+n_2)-1} \varepsilon\leq \sqrt[4]{17}^{-n-1}\varepsilon$. Hence, $x\ast y\ast (-x)\ast(-y)\in B_n$, concluding $[a,b]\in U_n$. 
\item Since $\exp$ is a diffeomorphism in $U_0^2\subseteq \widetilde{U}$, we get that $g^2=h^2$ if and only if $2\log(h)=\log(h^2)=\log(g^2)=2\log(g)$, if and only if $g=h$.
\item Note that $g\in U_{n+4}$ if and only if $\log(g)\in B_{n+4}$, if and only if $\log(g^{17})=17\log(g)\in B_n$, if and only if $g^{17}\in U_n$. \qedhere
\end{enumerate}
\end{proof}

\begin{coro} \label{c:corollary 3 general}
Fix $k,l,N\in \N$ and $s\map \N\to \N$. There is $m\coloneqq m(k,l,N,s)\in\N$ such that the following holds: 

Let $G$ be a metric group, $X$ an $(l,r)${\hyp}Lipschitz symmetric subset and $r_0,\ldots,$ $r_m$ with $2r_i\leq r_{i-1}$ such that 
\[\Nn_{r_i}(X^9)\leq k\cdot \Nn_{9r_i}(X)<\infty.\]
Then, for some $c\leq m$, there is a sequence $X_N\subseteq \cdots\subseteq X_0\subseteq X^8$ satisfying the following properties for each $n<N$ with $s\coloneqq s(c)$: 
\[\begin{array}{ll}
\mathrm{(1)}& X^2\mathrm{\ and\ }X_0\mathrm{\ are\ }(c,r_s)\mbox{\hyp}\mathrm{commensurable.}\\ 
\mathrm{(2)} & X_{n+1}X_{n+1}\subseteq \D_{r_s}(X_n).\\
\mathrm{(3)} & X_n\mathrm{\ is\ covered\ by\ }c\mathrm{\ translates\ of\ }\D_{r_s}(X_{n+1}).\\
\mathrm{(4)} & x^{-1}X_{n+1}x\subseteq \D_{r_s}(X_n)\mathrm{\ for\ every\ }x\in X_1.\\
\mathrm{(5)} & [X_{n_1},X_{n_2}]\subseteq \D_{r_s}(X_n)\mathrm{\ whenever\ }n< n_1+n_2.\\
\mathrm{(6)} & \{x\in X_0\sth x^{17}\in X_n\}\subseteq X_{n+1}.\\
\mathrm{(7)} &\mathrm{If\ }x,y\in X_0\mathrm{\ with\ }x^2=y^2,\mathrm{\ then\ }y^{-1}x\in \D_{r_s}(X_N).\end{array}\]
\end{coro}
\begin{proof} Aiming a contradiction, suppose otherwise. Thus, we get an $l${\hyp}Lipschitz sequence of growth $k$ in doubling scales $(G_m,X_m,r_{i,m})_{i\leq m\in\N}$ of counterexamples. Let $(G^*,X,\ldots)$ be a non{\hyp}principal ultraproduct as in \cref{t:main theorem}. Write $G$ for the $\bigvee_0${\hyp}definable subgroup generated by $X$. By \cref{t:main theorem}, we have a Lie model $\pi\map H\to L\coloneqq\faktor{H}{K}$ of $G\cdot o_r(1)$ with $o_r(1)\subseteq K\subseteq X^8\cdot o_r(1)$ such that $H$ is generated by $H\cap X^8\cdot o_r(1)$, $H\cap X^4$ is $o_r(1)${\hyp}roughly commensurable to $X^2$ and $\pi(H\cap X^8)$ is a neighbourhood of the identity in $L$.

By \cref{l:lemma corollary 3}, for any given compact neighbourhood of the identity $U_{-1}$, it is possible to find a sequence $(U_n)_{n\in\N}$ of symmetric compact neighbourhoods of the identity with $U_{n+1}\subseteq U_n^\circ$ and $U_0\subseteq U_{-1}$ satisfying the following properties:
\[\begin{array}{ll}
\mathrm{(1)}& U_{-1}\mathrm{\ and\ }U_0\mathrm{\ are\ }c_0\mbox{\hyp}\mathrm{commensurable.}\\ 
\mathrm{(2)} & U_{n+1}U_{n+1}\subseteq U_n.\\
\mathrm{(3)} & U_n\mathrm{\ is\ covered\ by\ }c_0\mathrm{\ translates\ of\ }U_{n+1}.\\
\mathrm{(4)} & x^{-1}U_{n+1}x\subseteq U_n\mathrm{\ for\ every\ }x\in U_0.\\
\mathrm{(5)} & [U_{n_1},U_{n_2}]\subseteq U_n\mathrm{\ whenever\ }n\leq n_1+n_2.\\
\mathrm{(6)} & \{x\in U_0\sth x^{17}\in U_n\}=U_{n+4}.\\
\mathrm{(7)} &\mathrm{If\ }x,y\in U_0\mathrm{\ with\ }x^2=y^2,\mathrm{\ then\ }x=y.\end{array}\]
In our case, we pick $U_{-1}=\pi(H\cap X^8)$. 

Now, as $\pi\map H\to L$ is continuous from a logic topology using enough parameters, there are symmetric relatively definable subsets $\pi^{-1}(U_{n+1})\subseteq Y_n\subseteq \pi^{-1}(U_n)$. Thus, $(Y_n)_{n\in\N}$ is a sequence of relatively definable symmetric subsets with $Y_{n+1}\subseteq Y_n$ and $Y_0\subseteq X^8\cdot o_r(1)$ satisfying the following properties: 
\[\begin{array}{ll}
\mathrm{(1)}& X^2\cdot o_r(1)\mathrm{\ and\ }Y_0\mathrm{\ are\ }c_1\mbox{\hyp}\mathrm{commensurable.}\\ 
\mathrm{(2)} & Y_{n+2}Y_{n+2}\subseteq Y_n.\\
\mathrm{(3)} & Y_n\mathrm{\ is\ covered\ by\ }c_1\mathrm{\ translates\ of\ }Y_{n+1}.\\
\mathrm{(4)} & y^{-1}Y_{n+2}y\subseteq Y_n\mathrm{\ for\ every\ }y\in Y_0.\\
\mathrm{(5)} & [Y_{n_1},Y_{n_2}]\subseteq Y_n\mathrm{\ whenever\ }n<n_1+n_2.\\
\mathrm{(6)} & \{y\in Y_0\sth y^{17}\in Y_n\}\subseteq Y_{n+3}.\\
\mathrm{(7)} &\mathrm{If\ }x,y\in Y_0\mathrm{\ with\ }x^2=y^2,\mathrm{\ then\ }x=y.\end{array}\]

Take $X'_{n+1}=Y_n\cap X^8$. Obviously, $X'_n$ is a symmetric subset. As $Y_n$ is relatively definable in $X^8\cdot o_r(1)$ and $X^8\subseteq X^8\cdot o_r(1)$ is definable, we conclude that $X'_{n+1}$ is definable. Now, note that $Y_n\subseteq X'_n\cdot o_r(1)$. Indeed, given $y\in Y_n$, we know $y\in \pi^{-1}(U_n)$. Thus, $o_r(y)\subseteq \pi^{-1}(U_n)\subseteq Y_{n-1}$, so $X^8\cap o_r(y)\subseteq X'_n$. As $Y_n\subseteq X^8\cdot o_r(1)$, we get that $y\in X^8\cdot o_r(1)$, so there is $x\in o_r(y)\cap X^8\subseteq X'_n$. Then, $y\in o_r(x)\subseteq X'_n\cdot o_r(1)$. Since $y$ is arbitrary, we conclude $Y_n\subseteq X'_n\cdot o_r(1)$. Hence, $(X'_n)_{n\in\N_{>0}}$ is a sequence of symmetric definable subset with $X'_{n+1}\subseteq X'_n$ and $X'_1\subseteq X^8$ satisfying the following properties:
\[\begin{array}{ll}
\mathrm{(1)}& X^2\mathrm{\ and\ }X'_1\mathrm{\ are\ }o_r(1)\mbox{\hyp}\mathrm{rough\ }c\mbox{\hyp}\mathrm{commensurable.}\\ 
\mathrm{(2)} & X'_{n+3}X'_{n+3}\subseteq o_r(X'_n).\\
\mathrm{(3)} & X'_n\mathrm{\ is\ covered\ by\ }c\mathrm{\ translates\ of\ }o_r(X'_{n+1}).\\
\mathrm{(4)} & aX'_{n+3}a^{-1}\subseteq o_r(X'_n)\mathrm{\ for\ every\ }a\in X'_1.\\
\mathrm{(5)} & [X'_{n_1},X'_{n_2}]\subseteq o_r(X'_n)\mathrm{\ whenever\ }n<n_1+n_2-2.\\
\mathrm{(6)} & \{x\in X'_1\sth x^{17}\in X'_n\}\subseteq X'_{n+3}.\\
\mathrm{(7)} &\mathrm{If\ }x,y\in X'_1\mathrm{\ with\ }x^2=y^2,\mathrm{\ then\ }y^{-1}x\in \bigcap_{n\in\N} o_r(Y_n).\end{array}\]

Set $X_n=X'_{3n+1}$ and $s=s(c)$. We have a sequence $X_N\subseteq \cdots\subseteq X_1\subseteq X^8$ of symmetric definable subsets containing the identity satisfying the following properties:
\[\begin{array}{ll}
\mathrm{(1)}& X^2\mathrm{\ and\ }X_0\mathrm{\ are\ }(c,r_s)\mbox{\hyp}\mathrm{commensurable.}\\ 
\mathrm{(2)} & X_{n+1}X_{n+1}\subseteq \D_{r_s}(X_n).\\
\mathrm{(3)} & X_n\mathrm{\ is\ covered\ by\ }c\mathrm{\ translates\ of\ }\D_{r_s}(X_{n+1}).\\
\mathrm{(4)} & x^{-1}X_{n+1}x\subseteq \D_{r_s}(X_n)\mathrm{\ for\ every\ }x\in X_1.\\
\mathrm{(5)} & [X_{n_1},X_{n_2}]\subseteq \D_{r_s}(X_n)\mathrm{\ whenever\ }n< n_1+n_2.\\
\mathrm{(6)} & \{x\in X_0\sth x^{17}\in X_n\}\subseteq X_{n+1}.\\
\mathrm{(7)} &\mathrm{If\ }x,y\in X_0\mathrm{\ with\ }x^2=y^2,\mathrm{\ then\ }y^{-1}x\in \D_{r_s}(X_N).\end{array}\]
By {\L}o\'{s}'s theorem, this contradicts the initial assumption that $(G_m,X_m,(r_{i,m})_{i\leq m})$ is a counterexample.
\end{proof}

\section{Digression on de Saxc\'{e} product theorem} 

In \cite{rodriguez2020piecewise}, the close relationship between the model theoretic Generic Set Lemma \cite[Theorem 2.1]{rodriguez2020piecewise} and the additive combinatorics Machado's Closed Approximate Subgroups Theorem \cite[Theorem 1.4]{machado2023closed} was already discussed; both are special cases of a more general topological result. Here, we take the opportunity to explain the further relation of these two results with de Saxc\'{e}'s Product Theorem \cite[Theorem 1]{desaxce2014product}, which also has some underlying similarities with the main results of this paper. First of all, we note the following remarkable consequences of Machado's Closed Approximate Subgroup Theorem by using Poguntke's Theorem \cite[Theorem 3.3]{poguntke1994dense}. Recall that a \emph{semi{\hyp}simple Lie group} is a connected Lie group that has no non{\hyp}trivial connected solvable normal subgroups. A \emph{simple Lie group} it is a non{\hyp}abelian connected Lie group that has no non{\hyp}trivial connected normal subgroups.  

\begin{lem}\label{l:poguntke} Let $G$ be a semi{\hyp}simple Lie group and $X$ a closed approximate subgroup with empty interior: 
\begin{enumerate}[label={\emph{(\arabic*)}}, noitemsep, wide]
\item Suppose $X$ is compact. Then, $X$ is contained in finitely many translates of a proper connected closed subgroup $H\lneq G$.
\item Suppose $X$ is connected. Then, $X$ is contained in a proper connected closed subgroup $H\lneq G$.
\item Suppose $G$ is a simple Lie group. Then, $X$ is discrete or contained in a proper closed subgroup $H\lneq G$.
\end{enumerate} 
\end{lem}
\begin{proof} By Machado's Closed Approximate Subgroups Theorem \cite[Theorem 1.4]{machado2023closed}, there is a subgroup $L\leq G$ containing $X$ that admits a Lie group structure such that $X$ has non{\hyp}empty interior, the map $\inc\map L\to G$ is a $1${\hyp}to{\hyp}$1$ continuous group homomorphism and $X$ has the same subspace topology induced by $L$ and $G$. By \cite[Theorem 3.3]{poguntke1994dense} and \cite[Proposition 6.5]{lee2002introduction}, as $G$ is semi{\hyp}simple, $\inc_{\mid L^{\circ}}$ has dense image if and only if it is actually an isomorphism. Since $X\cap L^{\circ}$ has empty interior in $G$ but not in $L$, we conclude that $H\coloneqq\overline{\inc(L^{\circ})}$ is a proper connected closed subgroup of $G$. 
\begin{enumerate}[label={\rm{(\arabic*)}}, wide]
\item If $X$ is compact, finitely many translates of $L^{\circ}$ cover it, so finitely many translates of $H$ cover $X$.
\item If $X$ is connected, $X\subseteq L^{\circ}$, so $X\subseteq H$.
\item Finally, suppose $G$ is simple and $X$ is not contained in a proper closed subgroup of $G$. Thus, the group generated by $X$ is dense in $G$, so $\inc$ has dense image. Since $L^{\circ}$ is normal in $L$, so is $H$ in $\inc(L)$. As $\inc(L)$ is dense and $H$ is closed, we conclude that $H$ is a normal subgroup of $G$. Since $G$ is simple, we conclude that $H$ is trivial, so $L^{\circ}$ is trivial, concluding that $L$ is discrete. Hence, $X$ is discrete.\qedhere
\end{enumerate}
\end{proof}

As a consequence, by an easy ultraproduct argument, we get the following variations of de Saxc\'{e}'s Product Theorem \cite[Theorem 1]{desaxce2014product}. Let us restate first the original result by de Saxc\'{e}.  
\begin{fact}[de Saxc\'{e}'s Product Theorem] Let $G$ be a simple Lie group and take some left invariant metric. There is a neighbourhood $U$ of the identity such that, for any $0<\sigma<\dim(G)$, there are $\varepsilon=\varepsilon(\sigma)>0$ and $\delta_0=\delta_0(\sigma)$ with the following property: if $X\subseteq U$ and $0<\delta<\delta_0$ satisfy 
\begin{enumerate}[label={\emph{(\arabic*)}}, wide]
\item $\Nn^\cov_\delta(X)\leq \delta^{-\sigma-\varepsilon}$,
\item $\Nn^\cov_\rho(X)\geq \delta^{\varepsilon}\rho^{-\sigma}$ for any $\rho\geq\delta$ and
\item $\Nn^\cov_\delta(X^3)\leq \delta^{-\varepsilon}\Nn^\cov_\delta(X)$,
\end{enumerate}
then there is a closed connected subgroup $H\lneq G$ with $X\subseteq \D_{\delta^{\varepsilon}}(H)$.
\end{fact}
The third hypothesis in \cite[Theorem 1]{desaxce2014product} is very close to assuming that $X$ is a metric approximate subgroup, while the other two hypotheses are morally saying that the dimension of $X$ looks like $0<\sigma<d$. Our variations are based on these analogies:

\begin{prop}[A Product Theorem for Simple Lie Groups] \label{p:de saxce} Let $G$ be a simple Lie group and $U$ a compact neighbourhood of the identity. Take some left invariant metric. Let $0<\sigma_0<\sigma_1<\dim(G)$ and $c_0,c_1,k,s\in \N$. Then, there is $m\in\N$ such that, for any $(k,2^{-m})${\hyp}metric approximate subgroup $X\subseteq U$ satisfying $2^{i\sigma_0-c_0}\leq\Nn_{2^{-i}}(X)\leq 2^{i\sigma_1+c_1}$ for each $i\leq m$, there is a closed subgroup $H\lneq G$ with $X$ covered by $\D_{2^{-s}}(H)$.
\end{prop}
\begin{proof} Aiming a contradiction, suppose otherwise. Then, we have a counterexample $X_m$ for each $m\in\N$. Take an ultraproduct in the sense of (unbounded) continuous logic, i.e. take an ultraproduct, take the subgroup generated by $U$ and quotient by the infinitesimals. By compactness of $U$, we end then with a closed subset $X\subseteq U$ of $G$. By {\L}o\'{s}'s Theorem, $X$ is a $k${\hyp}approximate subgroup and satisfies $2^{i\sigma_0-c_0}\leq \Nn_{2^{-i}}(X)\leq 2^{i\sigma_1+c_1}$ for every $i\in\N$. Then, $0<\sigma_0\leq \dim(X)\leq \sigma_1<d$, where $\dim$ denotes the (large) inductive dimension, by \cite[Theorem VII 2]{hurewicz2015dimension} and \cite[Eq.3.17 p.46]{falconer1990fractal}. Since the inductive dimension of discrete spaces is $0$, $X$ is not discrete. Also, by \cite[Corollary 1 of Theorem IV 3]{hurewicz2015dimension}, $X$ has empty interior in $G$. By \cref{l:poguntke}, as $G$ is simple, $X$ is contained in a proper closed subgroup $H\lneq G$. Thus, by {\L}o\'{s}'s Theorem, there is some $m$ such that $X_m\subseteq \D_{2^{-s}}(H)$, getting a contradiction with our initial assumption.
\end{proof}

\begin{prop}[A Product Theorem for Semisimple Lie Groups] \label{p:de saxce semisimple} Let $G$ be a semi{\hyp}simple Lie group and $U$ a compact neighbourhood of the identity. Take some left invariant metric. Let $\sigma<\dim(G)$, $C,k,s\in \N$ and $s\map \N\to \N$. Then, there is $m\in\N$ such that, for any $(k,2^{-m})${\hyp}metric approximate subgroup $X\subseteq U$ satisfying $\Nn_{2^{-i}}(X)\leq C\cdot 2^{i\sigma}$ for each $i\leq m$, there is $c\leq m$ and a connected closed subgroup $H\lneq G$ with $X$ covered by $c$ many translates of $\D_{2^{-s(c)}}(H)$.
\end{prop}
\begin{proof} Aiming a contradiction, suppose otherwise. Then, take an ultraproduct in the sense of (unbounded) continuous logic of a sequence of counterexamples. By compactness of $U$, we end then with a closed subset $X\subseteq U$ of $G$. By {\L}o\'{s}'s Theorem, $X$ is a $k${\hyp}approximate subgroup and satisfies $\Nn_{2^{-i}}(X)\leq C\cdot 2^{i\sigma}$ for every $i\in\N$. Then, $\dim(X)\leq \sigma<d$, where $\dim$ denotes the (large) inductive dimension, by \cite[Theorem VII 2]{hurewicz2015dimension} and \cite[Eq.3.17 p.46]{falconer1990fractal}. In particular, $X$ has empty interior in $G$ by \cite[Corollary 1 of Theorem IV 3]{hurewicz2015dimension}. By \cref{l:poguntke}, as $G$ is semi{\hyp}simple, $X$ is contained in finitely many translates of a proper connected closed subgroup $H\lneq G$. Thus, by {\L}o\'{s}'s Theorem, there is some $m>c$ such that $X_m$ is covered by $c$ many translates of $\D_{2^{-s(c)}}(H)$, getting a contradiction with our initial assumption.
\end{proof}

\bibliographystyle{alphaurl}
\bibliography{metric_approximate_subgroups}

\end{document}